\documentclass[12pt,a4paper]{article} %modify affilation?!!!! Federal State Budgetary Educational Institution of Higher Education "Saint-Petersburg State University?!! Less "? Shorter abstract?!
\usepackage{amsmath,amssymb,amsthm,amsfonts,amscd,euscript,verbatim, t1enc, newlfont, graphicx, cancel}
\usepackage{hyperref}
\usepackage[all]{xy}
\providecommand{\keywords}[1]{\textbf{\textit{Key words and phrases }} #1}
\providecommand{\subjclass}[1]{\textbf{\textit{2010 Mathematics Subject Classification.}} #1}

\theoremstyle{definition}
\newtheorem{theo}{Theorem}[subsection]

\newtheorem{pr}[theo]{Proposition}

 \newtheorem{lem}[theo]{Lemma}
 \newtheorem{coro}[theo]{Corollary}
  
\theoremstyle{remark}
\newtheorem{rema}[theo]{Remark}

\theoremstyle{definition}
\newtheorem{defi}[theo]{Definition}

\numberwithin{equation}{subsection}

\newcommand\cu{\underline{C}}
\newcommand\du{\underline{D}}

\newcommand\au{\underline{A}}
\newcommand\aucp{\underline{A}_{\cp}}
\newcommand\bu{\underline{B}}
\newcommand\hu{\underline{H}}

\newcommand\obj{\operatorname{Obj}}
\newcommand\mo{\operatorname{Mor}}
\newcommand\id{\operatorname{id}}
\DeclareMathOperator\adfu{\operatorname{AddFun}}

\DeclareMathOperator\kar{\operatorname{Kar}}
 
 \DeclareMathOperator\cok{\operatorname{Coker}}

\DeclareMathOperator\co{\operatorname{Cone}}
\DeclareMathOperator\ext{\operatorname{Ext}^{1}}
\DeclareMathOperator\extaucp{\operatorname{Ext}^{1}_{\aucp}}

\newcommand\hw{{\underline{Hw}}}
\newcommand\hrt{{\underline{Ht}}}

\newcommand\hf{{\underline{HF}}}

\newcommand\chowe{\underline{Chow}^{eff}}
\newcommand\chow{\operatorname{Chow}}
\newcommand\wchow{w_{Chow}{}}
\newcommand\wstu{w^{st}}
\newcommand\dm{DM}

\newcommand\da{DA^{et}(k,k)}
\newcommand\omdr{\mathbf{\Omega}}
\newcommand\omp{\mathbf{\Omega}'}

\newcommand\dmgm{DM^{gm}}

\DeclareMathOperator\cha{\operatorname{char}}

\newcommand\q{{\mathbb{Q}}}

\newcommand\z{{\mathbb{Z}}}

\newcommand\ob{^{-1}}
\newcommand\lam{\Lambda}

\newcommand\gam{\Gamma}
\newcommand\eps{\varepsilon}
\newcommand\ns{\{0\}}
\newcommand\ab{\operatorname{Ab}}
\newcommand\abfr{\operatorname{FreeAb}}
\newcommand\lvect{L-\operatorname{vect}}
\newcommand\ls{l_{S}}

\newcommand\cp{\mathcal{P}}
\newcommand\perpp{{}^{\perp}}
\newcommand\opp{^{op}}

\newcommand\shg{SH(G)}%index?? universe??!
\newcommand\wg{w^G}
\newcommand\shgl{SH_{\lam}(G)}%index?? universe??!
\newcommand\wgl{w_{\lam}^G}

\newcommand\macg{\mathcal{M}_G}
\newcommand\emg{\operatorname{EM}_G}
\newcommand\emo{\operatorname{EM}}
\newcommand\egam{\operatorname{EM}^{\gam}}

\newcommand\wsp{w^{sph}}
\newcommand\hwsp{\hw^{sph}}
\newcommand\shtop{SH}

\newcommand\hsing{H^{sing}}
\newcommand\hsingc{H_{sing}}

\newcommand\cuw{\cu_w}
\newcommand\kw{K_{\mathfrak{w}}}
\newcommand\ca{{\mathcal{A}}}
\newcommand\cacp{{\mathcal{A}_{\cp}}}

\begin{document}

\title
 {On morphisms killing weights and Hurewicz-type theorems}%weight complexes,  pure functors, and detecting  weights}
\author{Mikhail V. Bondarko
   \thanks{ %%!!!!
 The main results of the paper were  obtained under support of the Russian Science Foundation grant no. 16-11-00073.}}\maketitle
\begin{abstract} 
We study certain  "canonical weight decompositions" and apply the general theory to stable homotopical  and motivic examples. % slightly generalizing that defined by J. Wildeshaus. 
 
For a triangulated category $\cu$, any integer $n$, and a weight structure $w$ on $\cu$ %gives an {\it $n$-weight decomposition} 
 a triangle $LM\to M\to RM\to LM[1]$, where $LM$ is of weights at most $m-1$ and $RM$ is of weights at least $n+1$ for some $m\le n$, is determined by $M$ if exists. In this case we say (following J. Wildeshaus) that $M$ is {\it without weights $m,\dots,n$}. Since this happens if and only if the {\it weight complex} $t(M)\in \obj K(\hw)$  ($\hw$ is the {\it heart} of $w$) is homotopy equivalent to a complex with zero terms in degrees  $-n,\dots, -m$, this condition can be "detected" via {\it pure functors}. One can also take  $m=-\infty$ or $n=+\infty$  to obtain that the weight complex functor is "conservative and detects weights up to objects of infinitely small and infinitely large weights"; this is a significant improvement over previously known bounded conservativity results. % one may say that objects of $\cu$ may be detected by means of objects of a much simpler category $\kw(\hw)$ (this is a certain category of complexes over the heart $\hw$ of $w$, whereas %the latter 
 %$\hw$ is a category of Chow motives for Chow weight structures, and a certain additive orbit category for equivariant spherical ones). 
Applying this statement we %demonstrate that our results
 calculate intersections of certain purely generated subcategories and  %prove that certain weight-exact functors are conservative up to weight-degenerate objects
 study the conservativity of weight-exact functors. The main tool  is the new interesting notion of {\it morphisms killing weights $m,\dots, n$} that we study in detail as well.

We apply  general results to %to equivariant stable homotopy categories 
{\it spherical weight structures} for  $G$-spectra  (as introduced in the previous paper) and to Voevodsky motives. %We obtain that a $G$-spectrum is  without weights $m,\dots, n$ %if and only if 
 %whenever its %$RO(G)$-graded homology  $H^G_*$ (with values in Mackey functors) vanishes in these degrees; this is also equivalent to the vanishing of 
 %Bredon cohomology corresponding to $G$-Mackey functors vanishes in these degrees. % (here we take arbitrary Mackey functors and consider cohomology in degrees $m,\dots, n$). 
 %Moreover, the aforementioned conservativity of the weight complex functor results give 
%Moreover, we obtain a certain 
 This gives a converse to the (equivariant) stable Hurewicz theorem; %since in the case of a trivial group (and so, $\cu=\shtop$) the weight complex functor essentially calculates singular homology,
 in particular,  the singular homology of % a spectrum
  $E\in \obj \shtop$ vanishes in %non-positive 
 negative degrees if and only if  $E$ is an extension of a connective spectrum by an acyclic one. Moreover, the vanishing of $\hsing_i(E)$ for two subsequent values of $i$ gives a canonical (weight) "decomposition" of $E$ into a distinguished triangle as above. %????Furthermore, a $\shtop$-morphism $g$ kills weight $m$  if and only if $\hsingc^m(-,\gam)(g)=0$ for any abelian group $\gam$.
\end{abstract}
\subjclass{Primary 18E30   14C15 55P42  55N91; Secondary 55P91 18G25 18E40.}

\keywords{Triangulated category, weight structure, weight complex, weight-exact functor, conservativity, pure functors, equivariant stable homotopy category, Hurewicz theorems, Mackey functors, %Bredon cohomology, 
 motives.}

\tableofcontents

 \section*{Introduction}

Let us recall that for any object $M$ of a triangulated category $\cu$ and any integer $n$ a weight structure $w$ on $\cu$ gives (essentially by definition) an {$n$-weight decomposition} triangle $LM\to M\to RM\to LM[1]$, where $LM$ is of weights at most $n$ and $RM$ is of weights at least $n+1$;\footnote{I.e., $LM\in \cu_{w\le n}=\cu_{w\le 0}[n]$ and $RM\in \cu_{w\ge n+1}=\cu_{w\ge 0}[n+1]$.}  however, this triangle is (usually) not canonical. In particular, for the {\it spherical} weight structure on the stable homotopy category $\shtop$ (see \S4.2 of \cite{bwcp} and Theorem \ref{tsh}(\ref{itopass}) below) one can take $LM$ to be an arbitrary choice of an {\it $n$-skeleton} for the spectrum $M$ (in the sense of  \S6.3 of \cite{marg}); thus $LM$ is not determined by $M$. However, it was noticed by J. Wildeshaus (in %Definition 1.6 
 Proposition 1.7 of \cite{wild}) that  if one can choose an $n$-weight decomposition such that $LM$ is of weight at most $m-1$ for some $m\le n$ then this stronger assumption makes the decomposition canonical. In this case $M$ is said to be {\it without weights $m,\dots, n$}.  In this paper we prove that $M$ satisfies %(a minor modification of) 
  this condition if and only if its {\it weight complex} $t(M)$ is homotopy equivalent to a complex with zero terms in degrees  $-n,\dots, -m$;\footnote{Here one has to assume that $\cu$ is {\it weight-Karoubian}, i.e., that $\hw$ is idempotent complete; see %Proposition 
	 Theorem \ref{tpwckill}(\ref{iwckill4}) 
	%Proposition \ref{pwkar}(2) 
	 and \S\ref{sindwd}. However, this is a rather reasonable assumption since it is obviously fulfilled whenever $\cu$ is idempotent complete itself.}
	recall here that $t$ is a "weakly exact" functor from $\cu$ into a certain quotient $\kw(\hw)$ of the homotopy category of complexes in the {\it heart} $\hw$ of $w$.  %\footnote{It appears that $t$ can "usually" be "enhanced" to an  exact functor $t_{st}:\cu\to K(\hw)$; see Corollary 3.5 of \cite{sosnwc} and \S6.3 of \cite{bws}.}%say only once; delete below?! 
 It easily follows that  one can find out whether  $M$ is  without weights $m,\dots, n$ by applying {\it pure functors} (as introduced in \S2.1 of \cite{bwcp}; see Definition \ref{dpure} below) to $M$. %(both of these notions were defined in previous papers of the author). 
Moreover,  one can  put   $m=-\infty$ or $n=+\infty$ in these statements to obtain that %the weight complex functor 
 $t$ is "conservative up to ({\it weight-degenerate}) objects of infinitely small and infinitely large weights". This is a significant improvement over previously known bounded conservativity results;  one may say that objects of $\cu$ may be "detected" by means of objects of a much simpler category $\kw(\hw)$. We apply our conservativity of weight complexes result (and illustrate this yoga) to calculate certain intersections of purely generated subcategories (this result was %already
  applied in \cite{binters}) and to prove that certain weight-exact functors are conservative up to weight-degenerate objects. The latter statement generalizes Theorems 2.5 and 2.8 of \cite{wildcons} %the (more general) question of the conservativity of {\it weight-exact} functors was studied. We demonstrate (in \S\ref{sprtwcons} below) that our results imply a certain stronger statement of this form for 
 (in particular, to not necessarily bounded weight structures); we also explain that   our Proposition \ref{pwcons} can be applied to an interesting motivic functor that was mentioned in \cite[Remarks 3.1.4(2), 3.3.2(1)]{bwcp} and is closely related to \cite{ayoubcon}.

Moreover, we apply our general results to equivariant stable homotopy categories and {\it spherical weight structures} on them (as introduced in \S4 of \cite{bwcp}). %We obtain that a $G$-spectrum is without weights $m,\dots, n$ if and only if its $RO(G)$-graded homology  $H^G_*$ (with values in Mackey functors) in these degrees vanishes; this is also equivalent to the vanishing of Bredon cohomology corresponding to Mackey functors (here we take arbitrary Mackey functors and consider cohomology in degrees $m,\dots, n$). Moreover, the 
 Then the aforementioned conservativity of the weight complex functor results give a certain converse to the (equivariant) stable Hurewicz theorem. In particular, in the case of a trivial group (and so, $\cu=\shtop$) the weight complex functor essentially calculates singular homology; thus we obtain that the singular homology of a spectrum $E\in \obj \shtop$ vanishes in %non-positive 
 negative degrees if and only if  $E$ is an extension of a connective spectrum by an acyclic one. This statement appears to be completely new, since in all the previously existing formulations only the case where $E$ is {\it bounded below} was considered (see Theorem 2.1(i) of \cite{lewishur}, Proposition 7.1.2(f) of \cite{axstab}, and Theorem 6.9 of \cite{marg}).
Moreover, the vanishing of $\hsing_i(E)$ for two subsequent values of $i$ gives a canonical "decomposition" of $E$ into a distinguished triangle; this result (along with its equivariant generalization) appears to be completely new as well. %Furthermore, a $\shtop$-morphism $g$ kills weight  if and only if $\hsingc^m(-,\gam)(g)$ for any abelian group $\gam$.????

The main tool for obtaining these results is the new interesting notion of {\it morphisms killing weights $m,\dots, n$}; %that we study in detail as well.  
 for a morphism $g:M\to N$ this means that $g$ is "compatible with"  some morphism  $w_{\le n}M\to w_{\le m-1}N$. % (we write this as $g\in\mo_{\cancel{[m,n]}}\cu$).  
 We prove that this definition of killing weights  for $g$ is equivalent to several other ones. In particular, if $m=n$ then this condition is equivalent to an easily formulated property of $t(g)$;  hence an $\shtop$-morphism $g$ kills weight $m$ if and only if $\hsingc^m(-,\gam)(g)=0$ for any abelian group $\gam$. More generally, one may say that an $\shtop$-morphism $g$ kills weights $m,\dots, n$ whenever it sends $n$-skeleta into $m-1$-ones
(see  Proposition \ref{pkillw}(\ref{ikw7}) and Theorem \ref{tsh}(\ref{itopass})).

Let us now describe the contents of the paper; some more information of this sort can also be found in the beginnings of sections. 

 \S\ref{sold} is mostly dedicated to the recollection of the existing theory of weight structures; yet we also prove some new %properties of weight structures and 
 statements (including certain properties of homotopy categories of complexes).

In \S\ref{skw} we  define  morphisms killing weights $m,\dots, n$ and  objects without these weights; we also  study these notions in detail. In particular, we relate killing weights to weight complexes and pure functors; this gives a new conservativity of the weight complex functor result.

In \S\ref{snkar} we extend some of the results of the previous section to the case where $\hw$ (and hence also $\cu$) is not %(necessarily) 
 Karoubian (i.e., idempotents do not necessarily yield direct summands in it); so we formulate %Proposition \ref{pwkar} and 
 Theorem \ref{twkar} that is central for this paper. Moreover, the examples that we give demonstrate that the corresponding modifications of the formulations  from \S\ref{skw} (as made in %Proposition \ref{pwkar}(2) and 
 Theorem \ref{twkar}) are necessary. %ewp and ewn?!
 We also show that our results can be used to calculate certain intersections of purely generated subcategories %(this result was  %already  applied in \cite{binters})
  and to prove that certain weight-exact functors are "conservative up to weight-degenerate objects".

In \S\ref{smash} we apply our general results to the study of {\it purely compactly generated categories}. Next we consider %and to
  (equivariant) stable homotopy examples to obtain the aforementioned Hurewicz-type theorems along with several related statements. 

%?????!The author is deeply grateful to prof. J.P. May for his very useful answers concerning equivariant homotopy categories.
%The author is deeply grateful to the Max Planck Institut f\"ur Mathematik  for the support received and its hospitality during the work on this version.

The author is deeply grateful to the referees for  really actual comments. He is also very grateful to the Max Planck Institut f\"ur Mathematik  for the support received and its  hospitality during the work on this version.

\section{Weight structures: reminder}\label{sold}

In \S\ref{snotata} we introduce some %(mostly, categorical) 
notation and %definitions. 
conventions.

In \S\ref{ssws} we recall some basics on weight structures. The only new statement of this section is Proposition \ref{pbw}(\ref{icompidemp}); it is rather technical but quite important for this paper.

In \S\ref{sswc} we recall  some %basics on 
properties of  weight complex functors and study the properties of the weak homotopy equivalence relation for morphisms between complexes. %?????!! 

\subsection{Some (categorical) notation }\label{snotata} %and lemmas??

\begin{itemize}

\item Given a category $C$ and  $X,Y\in\obj C$  we will write
$C(X,Y)$ for  the set of morphisms from $X$ to $Y$ in $C$.

%\item For categories $C',C$ we write $C'\subset C$ if $C'$ is a full subcategory of $C$. {del}

\item %Given a category $C$ and  $X,Y\in\obj C$, 
We say that an object $X$ of $C$ is a {\it
retract} of $Y\in \obj \cu$ %(and $Y$ is {\it coretract} of $X$)
 if $\id_X$ can be %factorized as $X\stackrel{i}{\to} Y\stackrel{p}{\to}X$
 factored through $Y$. %\footnote{Clearly,  if $C$ is triangulated or abelian, then $X$ is a retract of $Y$ if and only if $X$ is its direct summand.}\ {del}

\item A %n additive  (not necessarily additive) 
 subcategory $\hu$ of an additive category $C$ 
%$\hu\subset C$ a subcategory $\hu$ is called
is said to be {\it retraction-closed} in $C$ if it contains all retracts of its objects in $C$.

\item  For any $(C,\hu)$ as above the full subcategory $\kar_{C}(\hu)$ of %an additive category
 $C$ whose objects
are all retracts of %objects of 
 (finite) direct sums of objects %of a subcategory 
$\hu$ in $C$ will be called the {\it retraction-closure} of $\hu$ in $C$; note that this subcategory is obviously additive and retraction-closed in $C$. 

\item Below $\au$ will always  denote some abelian category; $\bu$ is an additive category.

\item The {\it Karoubi envelope} $\kar(\bu)$ (no lower index) of %an additive category 
 $\bu$ is the category of ``formal images'' of idempotents in $\bu$. %\cite{bashli}!!
So, its objects are the pairs $(A,p)$ for $A\in \obj \bu,\ p\in \bu(A,A),\ p^2=p$, and the morphisms are given by the formula %(\ref{mthen}) below,
%\begin{equation}\label{mthen}
$$\kar(\bu)((X,p),(X',p'))=\{f\in \bu(X,X'):\ p'\circ f=f \circ p=f \}.$$ %\end{equation}
 The correspondence  $A\mapsto (A,\id_A)$ (for $A\in \obj \bu$) fully embeds $\bu$ into $\kar(\bu)$.
 Moreover, $\kar(\bu)$ is {\it Karoubian}, i.e.,   any idempotent morphism gives a direct sum decomposition in 
 $\kar(\bu)$. %Equivalently, $\bu$ is Karoubian if (and only if) the canonical embedding $\bu \to \kar (\bu)$ is an equivalence of categories. Recall also that $\kar(\bu)$ is triangulated if $\bu$ is (see  Theorem 1.5 of \cite{bashli}).

\item The symbol $\cu$ below will always denote some triangulated category; usually it will be endowed with a weight structure $w$. The symbols $\cu'$ and $\du$ will  also be used  for triangulated categories only.

\item For any  $A,B,C \in \obj\cu$ we will say that $C$ is an {\it extension} of $B$ by $A$ if there exists a distinguished triangle $A \to C \to B \to A[1]$.

%item Given a class $D$ of objects of $\cu$ we will write $\lan D\ra$ for the smallest  full retraction-closed triangulated subcategory of $\cu$ containing $D$. We will  call  $\lan D\ra$  the triangulated category {\it densely generated} by $D$. %strictly full triangulated subcategory of $\cu$ such that $D\subset  \obj \du$. We will  call  $\lan D\ra$  the triangulated category {\it strongly generated} by $D$. 
 %Certainly, this definition can be applied in the case $\du=\cu$.

\item For $X,Y\in \obj \cu$ we will write $X\perp Y$ if $\cu(X,Y)=\ns$. For
$D,E\subset \obj \cu$ we write $D\perp E$ if $X\perp Y$ for all $X\in D,\
Y\in E$.

%Given $D\subset\obj \cu$
Moreover,  we  will write $D^\perp$ for the class
$\{Y\in \obj \cu:\ X\perp Y\ \forall X\in D\}$;
%Sometimes we will denote by $D^\perp$ the corresponding full subcategory of $\cu$. 
dually, ${}^\perp{}D$ is the class
$\{Y\in \obj \cu:\ Y\perp X\ \forall X\in D\}$.

%\item Given $f\in\cu (X,Y)$, where $X,Y\in\obj\cu$, we will call the third verte of (any) distinguished triangle $X\stackrel{f}{\to}Y\to Z$ a {\it cone} of $f$.\footnote{Recall %{del}
%that different choices of cones are connected by non-unique isomorphisms.}\

\item %For an additive category $\bu$ 
We will write $K(\bu)$ for the homotopy category of (cohomological) complexes over $\bu$. Its full subcategory of
bounded complexes will be denoted by $K^b(\bu)$. We will write $M=(M^i)$ if $M^i$ are the terms of the complex $M$.
%???; $f^i$ will denote the $i$th component of a morphism of complexes $f$. 
%???? If we will say that an arrow (or a sequence of arrows) in $\au$ yields an object of $K^b(\bu)$, we will mean by default  that the last object of this sequence is in degree $0$. We will always extend a ``finite'' $\bu$-complex by $0$'s to $\pm \infty$	(to obtain an object of $K^b(\bu)\subset K(\bu)$).

	\item We will say that an additive covariant (resp. contravariant) functor from $\cu$ into $\au$ is {\it homological} (resp. {\it cohomological}) if it converts distinguished triangles into long exact sequences.
	
	For a (co)homological functor $H$ and $i\in\z$ we will write $H_i$ (resp. $H^i$) for the composition $H\circ [-i]$. % (resp. $H\circ [i]$).
%cohomological functors; compact objects???!
\end{itemize}

\subsection{Weight structures: basics}\label{ssws}

Let us recall some basic  definitions of the theory of weight structures. %  notion that is central for this paper.

\begin{defi}\label{dwstr}

I. A pair of subclasses $\cu_{w\le 0},\cu_{w\ge 0}\subset\obj \cu$ %(of {\it $w$-negative} and {\it $w$-positive} objects, respectively)
will be said to define a weight
structure $w$ on a triangulated category  $\cu$ if 
they  satisfy the following conditions.

(i) $\cu_{w\ge 0}$ and $\cu_{w\le 0}$ are %additive and 
retraction-closed in $\cu$ (i.e., contain all $\cu$-retracts of their objects).

(ii) {\bf Semi-invariance with respect to translations.}

$\cu_{w\le 0}\subset \cu_{w\le 0}[1]$, $\cu_{w\ge 0}[1]\subset
\cu_{w\ge 0}$.

(iii) {\bf Orthogonality.}

$\cu_{w\le 0}\perp \cu_{w\ge 0}[1]$.

(iv) {\bf Weight decompositions}.

 For any $M\in\obj \cu$ there
exists a distinguished triangle
$$LM\to M\to RM {\to} LM[1]$$
such that $LM\in \cu_{w\le 0} $ and $ RM\in \cu_{w\ge 0}[1]$.

Moreover, if $\cu$ is endowed with a weight structure then we will say that $\cu$ is a {\it weighted} (triangulated) category.
\end{defi}

We will also need the following definitions.

\begin{defi}\label{dwso}
Let $i,j\in \z$; assume that a triangulated category $\cu$ is endowed with a weight structure $w$.

\begin{enumerate}
\item\label{idh}
The full  subcategory $\hw$ of $ \cu$ whose objects are
$\cu_{w=0}=\cu_{w\ge 0}\cap \cu_{w\le 0}$ %and morphisms are $\hw(Z,T)=\cu(Z,T)$ for $Z,T\in \cu_{w=0}$,
 is called the {\it heart} of %the weight structure
$w$.

\item\label{id=i}
 $\cu_{w\ge i}$ (resp. $\cu_{w\le i}$, resp. $\cu_{w= i}$) will denote the class $\cu_{w\ge 0}[i]$ (resp. $\cu_{w\le 0}[i]$, resp. $\cu_{w= 0}[i]$).

\item\label{id[ij]}
$\cu_{[i,j]}$  denotes $\cu_{w\ge i}\cap \cu_{w\le j}$; so, this class  equals $\ns$ if $i>j$. 

%$\cu^b\subset \cu$ will be the category whose object class is $\cup_{i,j\in \z}\cu_{[i,j]}$; we will say that its objects are the {$w$-bounded} objects of $\cu$.

%\item\label{idbo} We will  say that $(\cu,w)$ is {\it  bounded}  if $\cu^b=\cu$ (i.e., if $\cup_{i\in \z} \cu_{w\le i}=\obj \cu=\cup_{i\in \z} \cu_{w\ge i}$).

\item\label{idwkar}
 We will say that %a triangulated category $\cu$ endowed with a weight structure $w'$ 
 $\cu$ (or $(\cu,w)$) is {\it weight-Karoubian} if $\hw$ is Karoubian. 

\item\label{idwe} %nafig?? later??!!
 Let %$\cu$ and 
  $\cu'$ be a triangulated category endowed with  a weight structure $w'$; let $F:\cu\to \cu'$ be an exact functor.

Then $F$ is said to be  {\it  weight-exact} (with respect to $w,w'$) if it maps $\cu_{w\le 0}$ into $\cu'_{w'\le 0}$ and
sends $\cu_{w\ge 0}$ into $\cu'_{w'\ge 0}$. 

\item\label{idrest}
Let $\du$ be a full triangulated subcategory of $\cu$.

We will say that $w$ {\it restricts} to $\du$ whenever the couple $(\cu_{w\le 0}\cap \obj \du,\ \cu_{w\ge 0}\cap \obj \du)$ is a weight structure on $\du$.

\item\label{ilrd} We will say that $M$ is left (resp., right) {\it $w$-degenerate} (or {\it weight-degenerate} if the choice of $w$ is clear) if $M$ belongs to $ \cap_{i\in \z}\cu_{w\ge i}$ (resp.    to $\cap_{i\in \z}\cu_{w\le i}$).

\item\label{iwnlrd} We will say that $w$ is left (resp., right) {\it non-degenerate} if all left (resp. right) weight-degenerate objects are zero.

We will say that $w$ is just {\it non-degenerate} if it is both left and right non-degenerate. %Define wd objects?!

\item\label{idbob} We will call $\cup_{i\in \z} \cu_{w\ge i}$ (resp. $\cup_{i\in \z} \cu_{w\le i}$) the class of {\it $w$-bounded below} (resp., {\it $w$-bounded above}) objects of $\cu$.
\end{enumerate}
\end{defi}

\begin{rema}\label{rstws}

1. A  simple (and still  useful) example of a weight structure comes from the stupid filtration on the homotopy category of cohomological complexes $K(\bu)$ for an arbitrary additive  $\bu$ (it can also be restricted to bounded complexes; see Definition \ref{dwso}(\ref{idrest})). In this case $K(\bu)_{\wstu\le 0}$ (resp. $K(\bu)_{\wstu\ge 0}$) is the class of complexes that are homotopy equivalent to complexes  concentrated in degrees $\ge 0$ (resp. $\le 0$); see Remark 1.2.3(1) of \cite{bonspkar} for more detail. %This notation will be 
We will use this notation below. 

 The heart of this weight structure $\wstu$ is the retraction-closure  of $\bu$  in  $K(\bu)$; hence it is equivalent to $\kar(\bu)$.  %(or in $K(\bu)$, respectively). 

2. A weight decomposition (of any $M\in \obj\cu$) is almost never canonical. 

Still for any $m\in \z$ the axiom (iv) gives the existence of a distinguished triangle \begin{equation}\label{ewd} w_{\le m}M\to M\to w_{\ge m+1}M\to (w_{\le m}M)[1] \end{equation}  with some $ w_{\ge m+1}M\in \cu_{w\ge m+1}$ and $ w_{\le m}M\in \cu_{w\le m}$; we will call it an {\it $m$-weight decomposition} of $M$.

 We will often use this notation below (even though $w_{\ge m+1}M$ and $ w_{\le m}M$ are not canonically determined by $M$); we will call any possible choice either of $w_{\ge m+1}M$ or of $ w_{\le m}M$ (for any $m\in \z$) a {\it weight truncation} of $M$. Moreover, when we will write arrows of the type $w_{\le m}M\to M$ or $M\to w_{\ge m+1}M$ we will always assume that they come from some $m$-weight decomposition of $M$.

3. In the current paper we use the ``homological convention'' for weight structures; 
it was previously used in  %\cite{wildshim}, 
\cite{wild}, \cite{wildcons}, \cite{bsnew}, %\cite{hebpo}, %\cite{bgern}, \cite{bkillw}, 
% \cite{brelmot}, \cite{bonivan}, 
\cite{bonspkar}, \cite{binters}, %\cite{bpgws}, 
and in \cite{bokum}, %and in \cite{bgn}, 
whereas in \cite{bws} %, \cite{bger}, 
 %and \cite{bontabu} 
 the ``cohomological convention'' was used. In the latter convention 
the roles of $\cu_{w\le 0}$ and $\cu_{w\ge 0}$ are interchanged, i.e., one
considers   $\cu^{w\le 0}=\cu_{w\ge 0}$ and $\cu^{w\ge 0}=\cu_{w\le 0}$. 
 
 We also recall that D. Pauksztello has introduced weight structures independently (in \cite{konk}); he called them
co-t-structures. %; in his papers (essentially) the cohomological convention is used.
\end{rema}

\begin{pr}\label{pbw}
Let  %$\cu$ be a triangulated category, $n\ge 0$, 
$m\le n\in\z$, $M,M'\in \obj \cu$, $g\in \cu(M,M')$. 

\begin{enumerate}
\item \label{idual}
The axiomatics of weight structures is self-dual, i.e., on $\cu'=\cu^{op}$
(so $\obj\cu'=\obj\cu$) there exists the (opposite)  weight structure $w^{op}$ for which $\cu'_{w'\le 0}=\cu_{w\ge 0}$ and $\cu'_{w'\ge 0}=\cu_{w\le 0}$.

\item\label{iort}
 $\cu_{w\ge 0}=(\cu_{w\le -1})^{\perp}$ and $\cu_{w\le 0}={}^{\perp} \cu_{w\ge 1}$.

\item\label{icoprod} $\cu_{w\le 0}$ is closed with respect to all coproducts that exist in $\cu$.
%nafig?!

\item\label{iext} %??!
 $\cu_{w\le 0}$, $\cu_{w\ge 0}$, and $\cu_{w=0}$ are additive. % and extension-closed. 

%\item\label{igenlm} The class $\cu_{[m,n]}$ is the %smallest %Karoubi-closed extension-stable subclass of $\obj\cu$ containing 
%extension-closure of $\cup_{m\le j\le n}\cu_{w=j}$.

%\item\label{ibond} If $M$ is bounded above (resp. below) and also left (resp. right) $w$-degenerate then it is zero. 

\item\label{icompl} %Let $ m\le l\in \z$, $M,M'\in \obj \cu$; fix certain weight decompositions
        %of $M[-m]$ and $M'[-l]$. Then  
				For any (fixed) $m$-weight decomposition of $M$ and an $n$-weight decomposition of $M'$  (see Remark \ref{rstws}(2))
				%any morphism
%one can extend
 $g$ can be extended %$g\in \cu(M, M')$ can be
to a %commutative diagram 
morphism of the corresponding distinguished triangles:
 \begin{equation}\label{ecompl} \begin{CD} w_{\le m} M@>{c}>>
M@>{}>> w_{\ge m+1}M\\
@VV{h}V@VV{g}V@ VV{j}V \\
w_{\le n} M'@>{}>>
M'@>{}>> w_{\ge n+1}M' \end{CD}
\end{equation}

Moreover, if $m<n$ then this extension is unique (provided that the rows are fixed).

%\item\label{itrun} Assume that $l>m$. Then for any  $l$-weight decomposition of $M$ and an $m$-weight decomposition of $w_{\le l}M$ the corresponding composed morphism $w_{\le m}(w_{\le l}M)\to w_{\le l}M\to M$ gives an $m$-weight decomposition of $M$.

 \item\label{iwd0} %For any weight decomposition of an
 If $M\in \cu_{w\ge m}$ %and $m\ge 0$ %(see (\ref{wd})) 
 %we have (see Remark \ref{rstws}(3)) 
 then $w_{\le n}M\in \cu_{[m,n]}$ (for any $n$-weight decomposition of $M$). %dualize?! Both?!
%Nado?!
Dually, if  $M\in \cu_{w\le n}$ 
 then $w_{\ge m}M\in \cu_{[m,n]}$.

\item\label{ifact} Assume %$M\in \cu_{w\le 0}$, $N\in  \obj \cu$, $g\in \cu(M,N)$. Then $g$ factors through $w_{\le 0}N$.
%$M\in \obj  \cu$,
 $M'\in   \cu_{w\ge m}$. Then %any $g\in \cu(M,M')$ 
 $g$ factors through $w_{\ge m}M$ (for any choice of the latter object).

Dually, if $M\in   \cu_{w\le m}$ then %any $g\in \cu(M,M')$ 
 $g$ factors through $w_{\le m}M'$.

\item\label{ikwkar} If $\cu$ is Karoubian then %$\cu$ is 
 it is also weight-Karoubian.

\item\label{icompidemp} %In the notation of the previous assertion assume that 
Assume that we are given a diagram of the form (\ref{ecompl}) and its rows are equal (so,
$M'=M$, $m=n$,  $w_{\le m} M= w_{\le n} M'$); also suppose that $g=\id_M$, $h$ is an idempotent endomorphism, and $(\cu,w)$ is weight-Karoubian. % whereas $\cu$ is Karoubian. 

Then there exists %an isomorphism  %for the 
 a decomposition $w_{\le m} M\cong M_1\bigoplus M_0$ %such that 
corresponding to $h$ (i.e., $h$ equals the projection of  $w_{\le m} M$ onto $M_1$). Moreover, %$\id_{M_1}$, 
we have $M_0\in \cu_{w=m}$, and the upper row of (\ref{ecompl}) 
 can be presented as the direct sum of (the corresponding two arrows in) a certain $m$-weight decomposition $M_1\to M\to M_2$ %and of %the distinguished triangle 
  with $(M_0\to 0 \to M_0[1])$.
\end{enumerate}
\end{pr}
\begin{proof}
%Just cite {bdetectw}?????!!!!!! Delete something?!
Assertions \ref{idual}--\ref{iwd0} %\ref%{icompl} % , \ref{icoprod}, \ref{iort}, \ref{iext}, \ref{igenlm}, %\ref{ibond}, %%\ref{iextcub}, \ref{icompl},   and \ref{iwd0}, %duality needed?!
%and \ref{iwe}
 were proved  in \cite{bws} (cf.  Remark 1.2.3(4) of \cite{bonspkar} and pay attention to Remark \ref{rstws}(3) above!),
whereas (the easy) assertion \ref{ifact} %and \ref{iwdmod} are 
 is given by Proposition 1.2.4(8) of \cite{bwcp}.
 %; see Remark 1.1.2(1), Proposition 1.3.3(3,6),  Proposition 1.3.6(1,2), Corollary 1.5.7, Proposition 1.5.6(2?!!!!),  Theorem 4.3.2(II),  and Proposition 5.2.2 of ibid., respectively.
%Proposition 1.5.6
%The remaining assertions also easily follow from the results of ibid. 
%??!

Assertion \ref{ikwkar} is obvious.

To prove assertion \ref{icompidemp} we take  a triangulated category $\cu'$ that is equivalent to $\kar(\cu)$ and contains $\cu$ as a full strict subcategory.  
Consider the decomposition $w_{\le m}M\cong M_1\bigoplus M_0$ corresponding to $h$ in %the triangulated category $\kar(\cu)$.  
 $\cu'$.  Since the diagram (\ref{ecompl}) is commutative, $c=c\circ h$; thus $c$ factors through $M_1$. 
Hence the distinguished triangle coming from the upper row of (\ref{ecompl}) can be decomposed into the direct sum of  the $\cu'$-distinguished triangle $M_0\to 0 \to M_0[1]$ with a %certain triangle
  $\cu'$-distinguished triangle  $M_1\to M\to M_2\to M_1[1]$. Thus $M_0$ is a retract of $w_{\ge m+1}M[-1]$ as well. Hence the morphism $\id_{M_0}$ factors through some morphism  $a: w_{\le m} M\to w_{\ge m+1}M[-1]$. Next, assertion \ref{ifact} implies that $a$ factors through $N=w_{\ge m}(w_{\le m} M)$; thus $M_0$ is a retract of $N$.  Now, $N$ belong to $\cu_{w=m}$ by assertion \ref{iwd0}. % of the proposition. 
 Since $\hw$ is Karoubian, we obtain that $M_0$ belongs to $\cu_{w=m}\subset \obj \cu$. If follows that $M_1$ and $M_2$ are objects of $\cu$ as well. Applying axiom (i) of Definition \ref{dwstr} we obtain that $M_1\in \cu_{w \le m}$ and $M_2\in \cu_{w \ge m+1}$; hence $M_1\to M\to M_2\to M_1[1]$ is an $m$-weight decomposition of $M$ by definition.
\end{proof}

\begin{rema}\label{rcompl}
Diagrams of the form (\ref{ecompl}) (also in the case $l<m$) are %very important 
 crucial for this paper. For any diagram of this sort we will say that the morphisms $h$ and $j$ are {\it $w$-truncations} of $g$.  

1. An important %particular case 
type of these diagrams is the one with $g=\id_M$ (for $M'=M$; cf. part \ref{icompidemp} of the proposition). Note that for $m<n$ the corresponding %connecting morphisms in  (\ref{ecompl}) 
 $w$-truncations of $\id_M$ are  unique (provided that the rows are fixed); if $m=n$ then we obtain a certain  (non-unique) "modification" of  an $m$-weight decomposition diagram.

%Another important observation (closely related to Proposition \ref{pbw}(\ref{iwd0})) is that in the case $m<l$ we have $\co(h)\in \cu_{[l+1,m]}$; see ???. %{bger}??

2. One  can "compose"  diagrams of the form (\ref{ecompl}), i.e., for any $q\in \cu(M',M'')$, $k\in \z$, and a morphism of triangles  of the form  
$$\begin{CD} w_{\le n} M'@>{}>>
M'@>{}>> w_{\ge n+1}M'\\
@VV{}V@VV{q}V@ VV{}V \\
w_{\le k} M''@>{}>>
M''@>{}>> w_{\ge k+1}M'' \end{CD}$$
one can compose its vertical arrows with the ones of (\ref{ecompl}) to obtain %a diagram 
a morphism of distinguished triangles 
$$ \begin{CD} w_{\le m} M@>{c}>>
M@>{}>> w_{\ge m+1}M\\
@VV{}V@VV{q\circ g}V@ VV{}V \\
w_{\le k} M''@>{}>>
M''@>{}>> w_{\ge k+1}M'' \end{CD}$$ 
Note that one does not have to assume $k\ge n$ here (and $n\ge m$ also is not necessary provided that the existence of  (\ref{ecompl}) is known in this case). Thus one may say that $w$-truncations of morphisms can be composed.

%Anyway,
 Moreover, if $k>m$ then %we obtain 
the composed diagram obtained this way is the only possible %diagram with these rows. 
morphism of triangles compatible with $q\circ g$.
%These observations surely can be applied if $g$ or $h$ equals $\id_{M'}$; in the case $m=l$ (resp. $l=k$) this corresponds to "modifying the $l$-decomposition of $M'$."

%???3. Note also that the diagram (\ref{ecompl}) can certainly be recovered from its left or right hand square.
\end{rema}

\subsection{On %weight Postnikov towers and 
weight complexes and weak homotopy equivalences}\label{sswc}

To define the weight complex functor we will need the following definition for complexes. Below  $\bu$ will always denote an additive category.

\begin{defi}\label{dbacksim}
Let $M$ and $N$ be objects of $K(\bu)$, and  $m_1,m_2\in C(\bu)(M,N)$.

1.  We %will 
write $m_1\backsim m_2$ if $m_1-m_2=d_Nx+yd_M$ for some collections of arrows
 $x^*,y^*\in \bu(M^*, N^{*-1})$, where $d_M$ and $d_N$ are the corresponding differentials.
We  call this relation the {\it weak homotopy equivalence} one. %relation. % (yet cf. Theorem 2.1 of \cite{barrabs}). %{bwcp}?!  %Remark \ref{rwc}(\ref{irwc2} below).%\footnote{This relation was earlier introduced in \cite{barrabs}; $m_1$ is {\it absolutely homologous} to $m_2$ in the terminology of that paper. Respectively, some of the results below concerning this equivalence relation were proved in ibid.}\ 

2. %More generally, 
Assume  $k\le l\in (\{-\infty\}\cup \z\cup \{+\infty\})$; also,  $k\in \z$ if $k=l$. 

Then we %will
  write $m_1\backsim_{[k,l]}m_2$ if $m_1-m_2$ is weakly homotopic to  $ m_0\in C(M,N)$ such that  $m_0^i=0$ for $k\le i \le l$ (and $i\in \z$).
\end{defi}

We need the following properties of these equivalence relations.

\begin{lem}\label{lwwh}
%In addition to the notation introduced above assume that $\bu$ is an additive category.
%Let $m_1,m_2\in C(\bu)(M,N)$ for some $\bu$-complexes $M,N$ (i.e., $m_1$ and $m_2$ are morphisms of complexes); $k\le l\in  (\{-\infty\}\cup \z\cup \{+\infty\})$.
Adopt the notation of  Definition \ref{dbacksim}.

\begin{enumerate}
\item\label{iwhecat}
%For any additive $\bu$ 
Factoring morphisms in $K(\bu)$ by the weak homotopy relation yields an additive category $\kw(\bu)$. Moreover, the corresponding full functor $K(\bu)\to \kw(\bu)$ is (additive and) conservative.

\item\label{iwhefu}
Let $\ca:\bu\to \au$ be an additive functor, where $\au$ is any abelian category, and assume that %. Then for any $N,N'\in \obj K(\bu)$ any pair of weakly homotopic morphisms 
$m_1$ is weakly homotopy equivalent to $m_2$. Then $m_1$ and $m_2$  induce equal morphisms of the homology $H_*(\ca(M^i))\to H_*(\ca(N^i))$.

Hence the correspondence $N\mapsto H_0(\ca(N^i))$ gives a well-defined functor $\kw(\bu)\to \au$.

\item\label{iwhefun} Applying an additive functor  $F:\bu\to \bu'$ to complexes termwisely one obtains an additive functor $\kw(F):\kw(\bu)\to \kw(\bu')$.   

\item\label{irwc3}  $m_1\backsim_{[k,l]}m_2$ if and only if $m_1\backsim_{[i,i]}m_2$ for any $i\in \z$ such that  $k\le i \le l$.

 Moreover, $m_1$ is weakly homotopy equivalent to $m_2$ if and only if $m_1\backsim_{[-\infty,+\infty]}m_2$. % for all 

\item\label{irwc4} If $k\in \z$ then $m_1\backsim_{[k,k]}0$ if and only if there exists $m_0\in C(\bu)(M,N)$ such that $m_1=m_0$ in  $K(\bu)(M,N)$ %(i.e., $m_1$ and $m_3$ are homotopy equivalent) 
  and  $m_0^k=0$. 

\item\label{irwc6}  $M$ belongs to $ K(\bu)_{\wstu\ge 0}$ %(for a $\bu$-complex $M$) 
 if and only if $\id_M \backsim_{[1,+\infty]}0_M$,   and $M\in K(\bu)_{\wstu\le 0}$ if and only if  $\id_M \backsim_{[-\infty,-1]}0_M$.
%\item\label{irwcwd} Assume that $m_1\backsim_{[k,l]}0$ and a $C(\bu)$-morphism $m$ becomes isomorphic to $m_1$ in the category $\kw(\bu)$. Then $m\backsim_{[k,l]}0$ as well.
%\item\label{ilwc0}  $A$ belongs  to $K(\bu)_{\wstu= 0}$ if and only if $A$ is a retract of $A^0$.  
\end{enumerate}
\end{lem}
\begin{proof}
Assertion \ref{iwhefun} is obvious, and the  remaining %assertions
 ones  are contained in Proposition B.2  of  \cite{bwcp}; cf. also \cite[\S3.1]{bws} for assertions \ref{iwhecat} and \ref{iwhefu}
\end{proof}

Now let us discuss the approach to weight complexes that we will use in the current paper.

\begin{rema}\label{rwc}
\begin{enumerate}
\item\label{irwch} %Similarly to \S3 of \cite{bws}, 
In the current paper we will use a certain additive weight complex functor $t:\cu\to \kw(\hw)$ (for any triangulated category $\cu$ endowed with a weight structure $w$). However, to define a canonical functor of this sort one has to replace $\cu$ by a certain equivalent (triangulated) category $\cuw$; see \S1.3 of \cite{bwcp} where this theory is exposed carefully (in contrast to \S3 of \cite{bws}; cf. %Appendix A.2
 Remark A.2.1(3) of \cite{bwcp}). Thus to define $t$ one should compose the (additive) "canonical weight complex functor" $t_{can}:\cuw\to \kw(\hw)$ with a splitting of the canonical equivalence $\cuw\to \cu$ (see Proposition 1.3.4 of \cite{bwcp}). Clearly (cf. Remark 5.3  of \cite{schnur}), any two splittings of this sort are isomorphic, and it is sufficient for our purposes to assume that one of them is chosen and so $t$ is fixed.

\item\label{irwcob} Moreover, we have no need to describe  weight complexes of all %objects and 
morphisms in $\cu$ explicitly; we prefer to list a collection of properties of $t$ instead.
So, we only sketch the description of $t(M)$ for $M\in \obj \cu$; the details can be found in loc. cit.

We choose (arbitrary) weight truncations $w_{\le n}M$ for all $n\in \z$, and take $g_{n}:w_{\le n}M\to w_{\le n+1}M$ to be the corresponding $w$-truncations of $\id_M$ (see Remark \ref{rcompl}). Denote $\co(g_n)$ by $M^{-n-1}[n+1]$;  it is easily seen that $M^i\in \cu_{w=0}$ for all $i\in \z$ and that the distinguished triangles coming from $g_n$ connect these objects to form a complex $(M^i)$. 

The problem with defining the functor $t$ is that this complex clearly depends on the choice of the objects $w_{\le n}M$.  However, for any choice of this form we have $(M^i)\cong t(M)$  (in $K(\hw)$; this isomorphism becomes canonical in $\kw(\hw)$). 

These observations easily imply Proposition \ref{pwt}(\ref{iwcons}) below; cf. also Remark \ref{rwchw}.
%(cf. Lemma \ref{lwwh}??!. 

%\item\label{irwcsh} Since %the aforementioned 
%category $\cuw$ is naturally endowed with a shift auto-equivalence, and both the aforementioned  functor $t_{can}:\cuw\to \kw(\hw)$ and the equivalence $\cuw\to \cu$ respect shifts, all the statements mentioning weight complexes can be naturally "shifted". %

%However, we will not apply this statement below. %avoid????!!!

\item\label{irwc7} %In "most of" the known examples  of weight structures 
 It appears that $t$ can "usually" be enhanced to an  exact (strong weight complex) functor $t^{st}:\cu\to K(\hw)$; see Corollary 3.5 of \cite{sosnwc}, \S6.3 of \cite{bws}, and Remark 1.3.5(3) of \cite{bwcp}.

Moreover, the author believes that the reader in (more or less, concrete) examples will not loose much if she assumes that $t^{st}$ exists throughout the paper. 

\item\label{irwc2} The weak homotopy equivalence relation was introduced  in \S3.1 of \cite{bws} independently from the earlier and closely related notion of {\it absolute homology}; cf. Theorem 2.1 of \cite{barrabs}. 

\item\label{irwcgs}
The term "weight complex" comes from \cite{gs}; yet the domain of the weight complex functor in that paper was not triangulated, whereas the target was ("the ordinary") $K^b(\chowe)$.
\end{enumerate}
\end{rema}

So we list the main properties of our weight complex functor $t: \cu\to \kw(\hw)$.

\begin{pr}\label{pwt}
Let $M,M'\in \obj \cu$, $g\in \cu(M,M')$ (where $\cu$ is %endowed with a weight structure $w$
 a weight triangulated category; see Definition \ref{dwstr}), and $h:M'\to \co(g)$ is the second side of a distinguished triangle containing $g$.

Then the following statements are valid.

\begin{enumerate}
\item\label{irwcsh} $t\circ [n]_{\cu}\cong [n]_{\kw(\hw)}\circ t$, where  $[n]_{\kw(\hw)}$ is the obvious shift by $[n]$ (invertible) endofunctor of the category $\kw(\hw)$.

\item\label{iwcex} 
There exists a lift of the $\kw(\hw)$-morphism chain $t(M)\stackrel{t(g)}{\to} t(M') \stackrel{t(h)}{\to} t(\co(g))$ to two sides of a distinguished triangle in $K(\hw)$.

\item\label{iwcons} If $M\in \cu_{w\le n}$ (resp. $M\in \cu_{w\ge n}$) then $t(M)$ belongs to $K(\hw)_{\wstu\le n}$ (resp. to $K(\hw)_{\wstu\ge n}$).

Moreover, if $M$ is left or right $w$-degenerate (see Definition \ref{dwso}(\ref{ilrd})) then $t(M)=0$.

\item\label{iwcfunct} Let $\cu'$ be a triangulated category endowed with a weight structure $w'$; let
 $F:\cu\to \cu'$ be a weight-exact functor. Then the composition $t'\circ F$ is isomorphic to $\kw(\hf)\circ t$, where 
$t'$ is a weight complex functor corresponding to $w'$, and the functor $\kw(\hf):\kw(\hw)\to \kw(\hw')$ is defined as in Lemma \ref{lwwh}(\ref{iwhefun}).

\item\label{iwcalc} %In the notation above
For any morphism of triangles 
\begin{equation}\label{ecalc}%??! =?!
 \begin{CD} w_{\le n-1} M@>{a}>>
w_{\le n} M@>{}>>\co(a)\\
@VV{c}V@VV{d}V@ VV{h}V \\
w_{\le n-1} M'@>{b}>>
w_{\le n}  M'@>{}>> \co(b)\end{CD}\end{equation}
where $a$, $b$, $c$, and $d$ are the corresponding $w$-truncations of $\id_M$, $\id_{M'}$, and $g$ (see Remark \ref{rcompl}), respectively, we have $\co(a),\co(b)\in \cu_{w=n}$, and $t(g)$ is isomorphic (as a $\kw(\hw)$-arrow) to a morphism $x$ whose $-n$th component $x^{-n}\in \mo(\hw)$ %modulo equivalence????!
 equals $h[-n]$. 

%Conversely?! for any $x^{-n}$?????!!!!!
Moreover, if $t(g)$ is isomorphic to a $\kw(\hw)$-morphism $y$ such that $y^{-n}=0$ then %one %can set $h=0$ in 
 %choose
 %for
  any choice of the rows in %the diagram
	 (\ref{ecalc})  can be completed to the whole diagram with $h=0$ in it.
%so that $h=0$ in it.
%the functor induced by the functor $\hf:\hw\to \hw'$.
\end{enumerate}
\end{pr}
\begin{proof}
Assertions \ref{irwcsh}--\ref{iwcfunct} are given by  Proposition 1.3.4(7,9,10,12)  %(7,8,10,12) 
 of  \cite{bwcp}. %\ref{iwcons} and \ref{iwcex} are given by  Proposition 1.3.4(7) of  \cite{bwcp}. %were essentially proved in \cite{bws}; see Proposition 1.3.4(7) and Remark 1.3.5(5) of  \cite{bwcp} for more detail. %Assertion \ref{iwcfunct}

%Taking into account our definitions
%Lastly,  The first part of assertion \ref{iwcalc} 
% \ref{iwcalc}.  The first part of the assertion 
Lastly, the first part of assertion \ref{iwcalc} follows from the definition of $t$ in ibid. %(see  %Proposition 1.3.4(6)  %and Remark 1.3.5(1)????  Applying an additive 
%of %that paper),
% part loc. cit.) 
and its "moreover" part %is an easy combination of Lemma \ref{lwwh}(\ref{irwcwd})  %nafig?????????? %with  
  easily follows from Proposition 1.3.4(13) of %\cite{bwcp} 
   ibid. (along with Lemma \ref{lwwh}(\ref{irwc4}) above).
%3.2.4(2) of \cite{bws}.
\end{proof}

\begin{rema}\label{rwchw}
Another property of $t$ that is easy to formulate is that its restriction to $\hw$ is isomorphic to the obvious embedding $\hw\to \kw(\hw)$; see %Remark 1.3.5(5) 
Proposition 1.3.4(10) of \cite{bwcp}. Note also that this property "almost follows" from part \ref{iwcalc} of our proposition.
\end{rema}

\section{On morphisms killing weights and objects without weights in a range}\label{skw}

In this section we introduce and study the main new notions of this paper.

In \S\ref{sskw} we %introduce our main notions of 
 define morphisms killing weights $m,\dots, n$ and  objects without these weights; we give several equivalent definitions of these notions.

In \S\ref{ssprkw} we establish several interesting properties of our notions. In particular, we prove that an object without weights  $m,\dots, n$ admits a (weight) {\it decomposition  avoiding these weights} (in the sense defined by J. Wildeshaus) if $\cu$ is weight-Karoubian.

In \S\ref{skwwc}  we relate killing weights to  the weight complex functor $t$. 
In particular, $M$ is without weights  $m,\dots, n$ if and only if  $t(M)$ %satisfies a similar 
 possesses this property.

In \S\ref{spure}  we relate killing a weight $m$ and object without weights in a range to {\it pure} functors (as introduced in \S2.1 of \cite{bwcp}).

\subsection{Morphisms that kill certain weights: equivalent definitions}\label{sskw}

\begin{pr}\label{pkillw}
Let $g\in \cu(M,N)$ (for some $M,N\in \obj \cu$); $m\le n\in \z$.
Then the following conditions are equivalent.

\begin{enumerate}
\item\label{ikw1}
There exists a choice of  $w_{\le n}M$ and $w_{\ge m}N$ such that the composed morphism $w_{\le n}M\to M\stackrel{g}{\to}N\to  w_{\ge m}N$ is zero (here the first and the third morphism in this chain come from the corresponding weight decompositions; see Remark \ref{rstws}(2)).  %, i.e., they are $w$-truncations of the $\id_M$ and).

\item\label{ikw3}
There exists a choice of  $w_{\le n}M$ and $w_{\le m-1}N$ and of a %commutative diagram 
morphism $h$ making the square
\begin{equation}\begin{CD} \label{ekw}
w_{\le n}M@>{}>>M
\\
@VV{%notation?!
h}V@VV{g}V\\
w_{\le m-1}N@>{}>>N 
\end{CD}\end{equation}
commutative.

\item\label{ikw5}
There exists a choice of  $w_{\ge n+1}M$ and $w_{\ge m}N$ and of a %commutative diagram 
morphism $j$ making the square
\begin{equation}\begin{CD} \label{ekw1}
M @>{}>>w_{\ge n+1}M
\\
@VV{%notation?!
g}V@VV{j}V\\
N @>{}>>w_{\ge m}N 
\end{CD}\end{equation}
commutative.
%\item\label{ikw6} For any choice of    $w_{\ge n+1}M$ and $w_{\ge m}N$ there exists a morphism $h$  that makes the diagram (\ref{ekw}) is commutative.

\item\label{ikw7} Any choice of an $n$-weight decomposition of $M$ and an $m-1$-weight decomposition of $N$ can be completed to a morphism of distinguished triangles of the form
\begin{equation}\label{eccompl} \begin{CD} w_{\le n} M@>{}>>
M@>{}>> w_{\ge n+1}M\\
@VV{h}V@VV{g}V@ VV{j}V \\
w_{\le m-1} N@>{}>>
N@>{}>> w_{\ge m}N \end{CD}\end{equation}

\item\label{ikw8} For any choice of  $m-1$- and $n$-weight decompositions of $M$ and $N$, and for $a$ and $b$ being the corresponding (canonical) connecting morphisms %$w_{\le m-1} N\to w_{\le n}  N$ 
$w_{\le m-1} M\to w_{\le n}  M$ and  $w_{\le m-1} N\to w_{\le n}  N$ respectively (see Remark \ref{rcompl}(1)),
there exists a commutative diagram
\begin{equation}\label{edouble}%??! =?!
 \begin{CD} w_{\le m-1} M@>{a}>>
w_{\le n} M@>{}>> M\\
@VV{c}V@VV{d}V@ VV{g}V \\
w_{\le m-1} N@>{b}>>
w_{\le n}  N@>{}>> N \end{CD}\end{equation}
along with a morphism $h\in \cu(w_{\le n} M, w_{\le m-1} N) $ that turns the corresponding "halves" of the left hand square of (\ref{edouble}) into commutative triangles.

\item\label{ikw9} %In the notation above
For any choice of the diagram (\ref{edouble}) as above its left hand commutative square can be completed to a morphism of  triangles as follows:
\begin{equation}\label{edoublecone}%??! =?!
 \begin{CD} w_{\le m-1} M@>{a}>>
w_{\le n} M@>{}>>\co(a)\\
@VV{c}V@VV{d}V@ VV{0}V \\
w_{\le m-1} N@>{b}>>
w_{\le n}  N@>{}>> \co(b)\end{CD}\end{equation}

\item\label{ikw9f} There exists a choice of  (\ref{edouble}) such that the corresponding diagram  (\ref{edoublecone})  gives a morphism of triangles.
\end{enumerate}
\end{pr}

\begin{proof}
%Assume that condition \ref{ikw3} is fulfilled. 
Conditions \ref{ikw1}, \ref{ikw3}, and \ref{ikw5} are equivalent by Proposition 1.1.9 of \cite{bbd} (that is easy; %one should consider
 in particular, the  long exact sequence $\dots\to \cu(w_{\le n}M,  w_{\le m-1}N)\to \cu(w_{\le n}M,  N)\to \cu(w_{\le n}M,  w_{\ge m}N)\to\dots$ yields that condition \ref{ikw1} is equivalent to \ref{ikw3}). 

Loc. cit. also implies %We also obtain
 that any of these conditions  yields the existence of some diagram of the form (\ref{eccompl}) for the corresponding choices of rows. One also obtains a diagram of this form for arbitrary choices of these weight decompositions by composing this diagram with the corresponding "change of weight decompositions" diagrams (see Remark \ref{rcompl}(1,2)); so we obtain condition \ref{ikw7}. On the other hand, the latter condition obviously implies  conditions \ref{ikw1}, \ref{ikw3}, and \ref{ikw5}.
One also obtains a diagram %of this
 of the form (\ref{eccompl})  for arbitrary choices of these weight decompositions by composing this diagram with the corresponding "change of weight decompositions" diagrams (see Remark \ref{rcompl}(1,2) once again); so we obtain condition \ref{ikw7}. On the other hand, the latter condition obviously implies  conditions \ref{ikw1}, \ref{ikw3}, and \ref{ikw5}.

Next, condition \ref{ikw8} clearly implies condition \ref{ikw3}. Conversely, %in order to %complete  (\ref{ekw})  to  (\ref{edouble}) 
 to obtain the commutative diagrams in condition \ref{ikw8} it suffices to take $a$ and $b$ to be the canonical connecting morphisms %$w_{\le m-1} N\to w_{\le n}  N$ 
$w_{\le m-1} M\to w_{\le n}  M$ and  $w_{\le m-1} N\to w_{\le n}  N$
(see Remark \ref{rcompl}(1)) respectively, $c=h\circ a$, and $d=b\circ h$. %unique from $h$?!

Next, condition  \ref{ikw9} clearly yields condition \ref{ikw9f}. %If condition \ref{ikw9f}
Now, consider the long exact sequence $\dots\to \cu(w_{\le n} M,w_{\le m-1} N)\to \cu(w_{\le n} M,w_{\le n} N) \to \cu(w_{\le n} M,\co (b))\to \dots$ (for an arbitrary choice of (\ref{edouble})). If condition \ref{ikw9f} is fulfilled, the composed morphism $w_{\le n} M\stackrel{d}{\to} w_{\le n} N\to \co (b)$ is zero; hence there exists a morphism $h\in \cu(w_{\le n} M,w_{\le m-1} N)$ making the corresponding triangle (a "half" of the left hand square in (\ref{edoublecone})) commutative. Combining this with the commutativity of the right hand square in (\ref{edouble})  we obtain condition \ref{ikw3} once again.

It remains to verify that   condition \ref{ikw8} implies condition \ref{ikw9}. The aforementioned long exact sequence gives the vanishing of the corresponding composed morphism $w_{\le n} M\to \co (b)$, whereas the long exact sequence $$\begin{gathered}
\dots\to \cu(w_{\le n} M,w_{\le m-1} N)\to \cu(w_{\le m-1} M,w_{\le m-1} N) \\ \to \cu(\co(a)[-1],w_{\le m-1} N)\to \dots \end{gathered}$$ yields the vanishing of the composed morphism $\co(a)[-1]\to w_{\le m-1} N$. We obtain that (\ref{edoublecone}) is a morphism of triangles indeed.
\end{proof}

%rcompl

Now we give the main original definitions of this paper (yet cf. Remark \ref{rwild}(\ref{irwild1}) below).

\begin{defi}\label{dkw}
Let $m\le n\in \z$.

1. We will say that  a morphism $g$ {\it kills weights $m,\dots, n$} if it satisfies the equivalent conditions of the previous proposition (and   we will say that $f$ {\it kills weight $m$} if $n=m$); denote the class of all $\cu$-morphisms killing weights  $m,\dots, n$ by $\mo_{\cancel{[m,n]}}\cu$.

2. We will say that an object $M$  of $\cu$  is {\it without weights $m,\dots,n$} (or that $M$ {\it avoids} these weights) if $\id_M$ kills weights  $m,\dots, n$;
 %We will denote
 the class of  $\cu$-objects without weights $m,\dots,n$ will be denoted by  $\cu_{w\notin[m,n]}$.
\end{defi}

\begin{rema}\label{rkwsd}
1. Obviously, these definitions are self-dual in the following natural sense: $g\in \mo_{\cancel{[m,n]}}\cu$ (resp.  $M\in \cu_{w\notin[m,n]}$) if and only if $g$ kills $w^{op}$-weights $-n,\dots,-m$ (resp. $M$ %is without 
 avoids $w^{op}$-weights  $-n,\dots,-m$) in $\cu'=\cu^{op}$ (see Proposition \ref{pbw}(\ref{idual})).

2. Now we describe a simple example that illustrates our definitions.

Let $\bu=\lvect$ (more generally, one can consider any  semi-simple  abelian category here) and recall that $\wstu$ denotes the stupid weight structure on $K(\bu)$ (see Remark \ref{rstws}(1); one can also take $\cu=K^b(\bu)$ here). Then  $M$ belongs to $ \cu_{\wstu\le 0}$ (resp. to $ \cu_{\wstu\ge 0}$) if and only if the homology $H_i(M)=H_0(M[-i])$ %(see the convention in \S\ref{snotata}) 
 vanishes for $i>0$ (resp. for $i<0$). Hence $g\in \cu(M,N)$ kills weights $m,\dots,n$ (resp. $M$ %is without 
 avoids weights  $m,\dots,n$) if and only if for the cohomological functor $H=\cu(-,L)$ (here we put $L$ in degree zero) we have $H^i(g)=0$ (resp. $H^i(M)=\ns$) 
%we have $\cu(-,L[i])(g)=0$ (resp. $\cu(M,L[i])=0$; so we put $L$ in degree $-i$) 
  for all $i\in \z$, $m\le i\le n$. Thus the functors %$\cu(-,L[i])$
	 $H^i$ for $m\le i\le n$ yield a collection of cohomology theories that detect whether $g\in \mo_{\cancel{[m,n]}}\cu$ %kills weights $m,\dots,n$ 
and $M\in \cu_{w\notin[m,n]}$. In the case  $m=n$ a certain  general result related to this one is given by  % related to this one %certain generalization of this observation  in 
 Proposition \ref{puredetkw}(I) below.
 % $M$ is without weights  $m,\dots,n$. %???????????? We %certainly 
% do not have so simple "detecting families" of functors in general; yet we will construct quite interesting detecting classes of cohomology below (see Theorem \ref{tkwhom}  for the general case and Proposition \ref{pshtopn}(\ref{it10})  for the case $\cu=\shtop$). We prefer considering cohomological detectors (in this paper) for the reasons explained in Remark \ref{rtst}(4)  (cf. also Remark \ref{rkwmot}(6) below).
\end{rema}

\subsection{Basic properties of our  notions}\label{ssprkw}

\begin{theo}\label{tprkw}
%Adopt the  notation of Proposition \ref{pkillw} and assume that $g$ .
Let  $M,N,O\in \obj \cu$, $h\in \cu(N,O)$, and assume that a morphism $g\in \cu(M,N)$ kills weights $m,\dots, n$ for some $m\le n\in \z$. Then the following statements are valid.

\begin{enumerate}
\item\label{iprkw1}
Assume  $m\le m'\le n'\le n$. Then %any $g$ killing weights $m,\dots, n$ 
$g$ also kills weights $m',\dots, n'$.

\item\label{iprkwds} $\mo_{\cancel{[m,n]}}\cu$ is closed with respect to %sums, 
 direct sums and retracts (i.e., $\bigoplus g_i$ kills weights $m,\dots, n$ if and only if all $g_i$ do that).

\item\label{iprkwid}  $\mo_{\cancel{[m,n]}}\cu$ is  a two-sided ideal of morphisms, i.e., for any $h'\in \cu(O,M)$ both $h\circ g$ and $g\circ h'$ kill weights $m,\dots, n$.

\item\label{iprkwcomp} Assume that %$h$ kills weights %$n+1,\dots, n'$ for some $n'>n$ 
%a $\cu$-morphism $h$ composable with $g$ 
$h$ kills weights $m',\dots, m-1$ for some $m'<m$. Then $h\circ g$ kills weights %$m,\dots,n'$ (resp. weights  
$m',\dots, n$.

\item\label{iprkwfunct} Let $F:\cu\to \du$ be a weight-exact functor (with respect to a certain weight structure for $\du$) and  assume that %a $\cu$-morphism 
 $h$ kills weights $m,\dots, n$. Then $F(h)$ kills these weights as well.

\item\label{iprkwfunctemb} For $F$ and $h$ as in the previous assertion assume that $F$ is a full embedding and $F(h)\in \mo_{\cancel{[m,n]}}\du$. %kills weights $m,\dots, n$. 
Then $h\in \mo_{\cancel{[m,n]}}\cu$. %kills these weights also.

\item\label{iprkwcompobj} Assume that $O$ %is without weights $m,\dots, n$ as well as without weights 
 avoids weights $m,\dots, n$ along with weights  $n+1,\dots, n'$ for some $n'>n$. Then $O\in  \cu_{w\notin[m',n]}$.

\item\label{iprkws} $O$ %is without 
 avoids weights $m,\dots, n$ if and only if $O$ is without weight $i$ whenever $m\le i\le n$. 

\item\label{iwildef1}
%Let 
Assume that there exists a distinguished triangle 
\begin{equation}\label{ewild}
X\to O \to Y\to X[1]
\end{equation}
with $X\in \cu_{w\le m-1}$, $Y\in \cu_{w\ge n+1}$ (we call it a {\it  decomposition avoiding weights  $m,\dots, n$} for $M$). Then (\ref{ewild})  gives  $l$-weight decompositions of $O$ for any $l\in \z$, $m-1\le l\le n$. Moreover, $O$ is without weights $m,\dots,n$, and  
this triangle is unique up to a canonical isomorphism. %functorial: here or in {rwild}?!

\item\label{iwildef2} Assume that $\cu$ is weight-Karoubian. Then the converse to the previous assertion is also valid (i.e., %for any object\
any  $O$  without weights $m,\dots,n$ %then there exists a distinguished triangle of the form (\ref{ewild})).
 admits a   decomposition avoiding weights  $m,\dots, n$).
\end{enumerate}
\end{theo}
\begin{proof}

\begin{enumerate}
\item Easy (if we use condition \ref{ikw3} of Proposition  \ref{pkillw}) %if we apply 
 from Remark \ref{rcompl}(1,2). %  (once again).  %?!!

\item  Proposition \ref{pbw}(\ref{iext}) implies that %a
  all direct sums of weight decompositions are weight decompositions. This implies the assertion easily; see %(if we use 
conditions \ref{ikw1} and \ref{ikw7} of Proposition  \ref{pkillw}. 

%\item,\item  
\ref{iprkwid}, \ref{iprkwcomp}.  Easy since we can compose the diagrams given by  Proposition  \ref{pkillw}(\ref{ikw3}) and Proposition \ref{pbw}(\ref{icompl}); see Remark \ref{rcompl}(2) once again.

%\item %If we have
\ref{iprkwfunct}. If we make choices corresponding to condition \ref{ikw1} of Proposition  \ref{pkillw}  and apply $F$ to %(some choice of) the vanishing for $h$ given by 
 the corresponding vanishing then we obtain this condition for $F(h)$.

%\item 
\ref{iprkwfunctemb}. For any choice of $w_{\le n}M$ and $w_{\ge m}N$  the composed morphism
$F(w_{\le n}M)\to F(M)\stackrel{F(h)}{\to}F(N)\to F( w_{\ge m}N)$ is zero (see  condition \ref{ikw3} of Proposition  \ref{pkillw}); hence  this condition is fulfilled for $h$.

%\item
\ref{iprkwcompobj}.  Since $\id_O\circ \id_O=\id_O$,  the statement follows from assertion \ref{iprkwcomp}.  %Immediate from the previous assertion (since $\id_O\circ \id_O=\id_O$).

%\item 
\ref{iprkws}.  If $O$ %is without 
 avoids weight $i$ whenever $m\le i\le n$ then iterating the previous assertion we obtain that $O$ is without weights $m,\dots, n$. Conversely, if $O$ satisfies the latter assumption  and  $m\le i\le n$ then $\id_O$ kills weight $i$ (i.e., $O$ %is without 
 avoids  weight $i$) %in this range
  according to assertion \ref{iprkw1}.

%\item
 \ref{iwildef1}. Each statement in this assertion easily follows from the previous ones.

(\ref{ewild}) gives the corresponding $l$-weight decompositions of $O$ just by definition. We obtain that $O$ %is without 
 avoids weights $m,\dots, n$ immediately (here we can use either condition \ref{ikw3} or condition \ref{ikw5} of Proposition  \ref{pkillw}). This triangle (\ref{ewild})  is canonical by Proposition \ref{pbw}(\ref{icompl}) (if we take $M=M'=O$, $g=\id_O$, $m=n-1$, and $l=n$ in it). %here we ignore

%\item 
\ref{iwildef2}. The idea is to "modify" any (fixed) $n$-decomposition  of $O$ using Proposition \ref{pbw}(\ref{icompidemp}).

We also fix an $m$-weight decomposition of $O$. According to condition \ref{ikw3} in  Proposition  \ref{pkillw} there exists a commutative square
%\begin{equation}
$$\begin{CD} %\label{ekw}
w_{\le n}O@>{}>>O
\\
@VV{%notation?!
z}V@VV{\id_O}V\\
w_{\le m-1}O@>{}>>O 
\end{CD}$$
Next, Proposition \ref{pbw}(\ref{icompl}) gives the existence and uniqueness of the square
$$\begin{CD} %\label{ekw}
w_{\le m-1}O@>{}>>O
\\
@VV{%notation?!
t}V@VV{\id_O}V\\
w_{\le n}O@>{}>>O 
\end{CD}$$
Now, we can consider multiple compositions of these squares (see Remark \ref{rcompl}). Hence the aforementioned uniqueness statement implies $t=t\circ z\circ t$. Thus the endomorphism $u=t\circ z$ is idempotent, and the square $$\begin{CD} %\label{ekw}
w_{\le n}O@>{}>>O
\\
@VV{%notation?!
u}V@VV{\id_O}V\\
w_{\le n}O@>{}>>O 
\end{CD}$$
is commutative. Now we apply Proposition \ref{pbw}(\ref{icompidemp});  for $X$ being the "image" of $u$ we obtain an $n$-weight decomposition 
$X\to O\to Y$. It remains to note that $X\in \cu_{w\le m-1}$ since $u$ factors through $w_{\le m-1}O$.
\end{enumerate} 
\end{proof}

\begin{rema}\label{rwild}
\begin{enumerate}
%\item\label{irwild0} Combining 
\item\label{irwild1}
%Certainly, the distinguished triangle (\ref{ewild}) yields an ?-weight decomposition of $M$. %needed in the proof anyway?!
The existence of  a   decomposition of $O$ avoiding weights  $m,\dots, n$ means precisely %yields that
 that $O$  is without weights $m,\dots,n$ in the sense of Definition 1.10 of \cite{wild}. So, our definition of this notion is equivalent to the original definition of Wildeshaus (who introduced this term) if $\cu$ is  weight-Karoubian;  recall here that this is automatically the case if $\cu$ is Karoubian (see Proposition \ref{pbw}(\ref{ikwkar})).  %however, 
 Still \S\ref{ssskwild} below demonstrates that this equivalence statement does not hold unconditionally. %In particular, 

Hence the uniqueness statement in Theorem \ref{tprkw}(\ref{iwildef1}) %is %a particular case of the functoriality result given by 
 coincides with Corollary 1.9 of \cite{wild}. % state it?! here?!
%Below we will also mention the functoriality of weight decompositions of this form. So we formulate 
Moreover, the obvious modification of the proof of   Theorem \ref{tprkw}(\ref{iwildef1})  gives %the corresponding 
 Proposition 1.7 of ibid. that is essentially as follows:  
 if $X_i\to O_i \to Y_i$ are distinguished triangles for $i=1,2$, $X_i\in \cu_{w\le m-1}$, $Y_i\in \cu_{w\ge n+1}$ (and $n\ge m$), then any $g:O_1\to O_2$ uniquely extends to a morphism of these weight decompositions (cf. Theorem \ref{tprkw}(\ref{iwildef1}) once again; note also that the argument used in the proof of our theorem extends to re-prove loc. cit. without any difficulty).  

\item\label{irwild2}
Combining parts \ref{iprkwds} and  \ref{iprkwid} of our theorem one immediately obtains that the sum of any two parallel morphisms killing weights $m,\dots, n$ kills these weights as well.

Moreover, part \ref{iprkwds} of the theorem implies that   $\cu_{w\notin[m,n]}$ is additive and retraction-closed in $\cu$, %??? (i.e., it is additive and retraction-closed in $\cu$), 
 whereas  part  \ref{iprkwid} yields that any $\cu$-morphism from $M$ kills weights $m,\dots,n$ if (and only if; look at $\id_M$)  $M$ avoids these weights.
%\item\label{irwild4}  Certainly, parts  \ref{iprkw1}--\ref{iprkwcomp} of  Theorem \ref{tprkw} imply that the sum of two morphisms $M\to N$ killing weights $m,\dots,n$ kills these weights also; a direct proof of this fact is  very easy as well.

\item\label{irwild6} %Certainly, p
 Part \ref{iprkwfunct} of our theorem says that weight-exact functors respect the condition of %being without 
 avoiding weights $m,\dots,n$, whereas   full weight-exact embeddings "strictly respect" this condition. Hence  full weight-exact embeddings of weight-Karoubian  categories also strictly respect the condition of an object to possess a decomposition  avoiding weights  $m,\dots, n$. %of the type (\ref{ewild}).  
\end{enumerate}
\end{rema}

\subsection{Relation to the weight complex functor} %(and its conservativity)}
\label{skwwc} 

Now we relate the properties studied in the previous subsection with the weight complex functor.

\begin{theo}\label{tpwckill}
Let $g\in \cu(M,N)$ (for some $M,N\in \obj \cu$); $m\le n\in \z$. Then the following statements are valid.

\begin{enumerate}
\item\label{iwckill1} $g$ kills weight $m$ if and only if $t(g)\backsim_{[-m,-m]} 0$ (in the notation of Definition \ref{dbacksim}). %Remark \ref{rwc}(\ref{irwc3}); 
 %recall that this property does not depend on the choice of $t(g)$ according to Lemma \ref{lbacksim}(\ref{irwcwd})). 

\item\label{iwckill2} If %$l\le m\in \z$ and 
$\{f_i\}$ for $n\ge i\ge m$ form a chain of composable $\cu$-morphism such that %$t(f_i)\backsim_{[-i,-i]}0$
 for all $i$ in this range, then %$f_{n}\circ f_{n-1}\circ\dots \circ f_m$???
  $f_m\circ\dots \circ f_{n-1}\circ f_{n}$ kills weights  $m,\dots, n$. 

\item\label{iwckill3} $M$ is without 
 %avoids
 weights $m,\dots, n$ if and only if  $t(\id_M)\backsim_{[-n,-m]}0$. %re-formulate in terms of WC-decompositions?!!

\item\label{iwckill4} Assume in addition that $\cu$ is weight-Karoubian. Then $M$ %is without 
 avoids weights $m,\dots, n$ if and only if  $t(M)$ is homotopy equivalent to a complex $C=(C^i)$ with $C^i=0$ for $-n\le i\le -m$. 
\end{enumerate}
\end{theo}
\begin{proof}

\begin{enumerate}
\item Immediate from Proposition \ref{pwt}(\ref{iwcalc}); see also condition \ref{ikw9f} in  Proposition \ref{pkillw}.
%This is just a re-formulation of 
%If $g$ kills weight $m$ then we can choose $t(g)$ such that the component $t(g)^{-m}$ is zero  (see Proposition \ref{pkillw}(\ref{ikw9})). Conversely, if $t(g)\backsim_{[-m,-m]} 0$ then $g$ fulfils the criterion of killing weight $m$ provided by Proposition \ref{pkillw}(\ref{ikw9f}) according to Proposition 3.2.4(2) of \cite{bws}. %Hence  $g$ kills weight $m$ (see Proposition \ref{pkillw}(\ref{ikw9f})).

\item Straightforward from the previous assertion combined with Theorem \ref{tprkw}(\ref{iprkwcomp}). % (see Remark \ref{rwc}(\ref{irwc3})).

\item If $M\in \cu_{w\notin[m,n]}$ %is without weights $m,\dots, n$
  then  %we can choose $t(\id_M)$ so that $t(\id_M)^i= 0$ for all $i$ between $-n$ and $-m$ (this is an easy consequence of Proposition \ref{pkillw}(\ref{ikw9})); hence %this is true for any choice of 
%Proposition \ref{pwt}(\ref{iwcalc}) easily implies that
 combining assertion \ref{iwckill1} with Lemma \ref{lwwh}(\ref{irwc3}) we obtain $t(\id_M)\backsim_{[-n,-m]}0$. Conversely, if  $t(\id_M)\backsim_{[-n,-m]}0$ then $t(\id_M)\backsim_{[i,i]} 0$ for all $i$ between $-n$ and $-m$; thus applying the previous assertion to 
%it is certainly without weight $i$ for any $i$ between $m$ and $n$. Hence applying the previous assertion to
 the composition $\id_{M}^{\circ n-m+1}$ we obtain that $M$ %is without 
 avoids weights $m,\dots, n$. 

\item The "if" implication follows from the previous assertion immediately; cf. Definition \ref{dbacksim}(2).

 Conversely, assume that $M$ %is without 
 avoids weights $m,\dots, n$.  By Theorem \ref{tprkw}(\ref{iwildef2}),  %if and only if
  $M$ possesses a  decomposition avoiding weights  $m,\dots, n$. Then for the corresponding objects $X$ and $Y$ (see (\ref{ewild})) Proposition \ref{pwt}(\ref{iwcons}) says that $t(X)\in K(\hw)_{\wstu\le m-1}$ and $t(Y)\in K(\hw)_{\wstu\ge n+1}$. Recalling the definition of $\wstu$ (in Remark \ref{rstws}(1)) and applying Proposition \ref{pwt}(\ref{iwcex}) we obtain the existence of a $K(\hw)$-distinguished triangle $T_X\to t(M)\to T_Y\to T_X[1]$, where $T_X$ and $T_Y$ have zero terms in degrees at most $-m$ and at least $-n$, respectively. Thus we obtain the "only if" implication.
\end{enumerate}
\end{proof}

\begin{rema}\label{rwildwc}

1. %Certainly, p
Part \ref{iwckill2} of our proposition is a vast generalization of (the nilpotence statement in) \cite[Theorem 3.3.1(II)]{bws}. %similar argument???
Moreover, we obtain an alternative proof of the latter statement that does not depend either on it or on Proposition 3.2.4 of ibid. (cf. Remark A.2.1(3) of \cite{bwcp}). 

%So 
 To justify these claims we recall that Theorem 3.3.1(II) of \cite{bws} essentially states that  $f=%f_{n}\circ f_{n-1}\circ\dots \circ f_m
  f_m\circ\dots \circ f_{n-1}\circ f_{n}=0$ whenever $f_i$ are certain morphisms between elements of $\cu_{[m,n]}$ (see Definition \ref{dwso}(\ref{id[ij]})) and $t(f_i)=0$. 
Now, if this is the case then clearly $t(f_i)\backsim_{[-i,-i]}0$; hence $f$ kills weights $m,\dots, n$   by %Proposition  
 Theorem \ref{tpwckill}(\ref{iwckill2}). Next, if $f\in \cu(M,N)$ for $M,N\in \cu_{[m,n]}$ then we can take $w_{\le n}M=M$ and $w_{\ge m}N=N$; thus $f=0$ (see condition \ref{ikw7} in Proposition \ref{pkillw}). Hence our %Proposition
 Theorem \ref{tpwckill}(\ref{iwckill2}) generalizes loc. cit. indeed.

%for $n\ge i\ge m$ is a chain of composable $\cu$-morphism such that $t(f_i)\backsim_{[-i,-i]}0$ for all $i$ in this range, then $f_{n}\circ f_{n-1}\circ\dots \circ f_m$ kills weights  $m,\dots, n$. 

2. %It could make  sense to
 One may extend the notion of morphisms killing weights (in a range) as follows: one can say that $g\in \mo \cu$  kills weights $m,\dots, n$ for any $m\le n$, where $m,n\in \{-\infty\}\cup \z\cup \{+\infty\}$, whenever $g$ kills weights $m',\dots, n'$ for any integers $m',n'$ such that  $m\le m'\le n'\le n$ (cf. Definition \ref{dbacksim}(2)). Respectively, one can  define objects without weights  $m,\dots, n$ in this extended range similarly to Definition \ref{dkw}(2). 

Most of the statements involving the corresponding definitions (in our paper) will remain true for their extended versions. A significant part of the resulting "infinite analogues" are proven in this text anyway; so we leave to the reader to track the parallels with the "finite cases" and to formulate those "infinite versions" that are not included here. However, it makes sense to realize that the "decompositions" provided by Theorems \ref{tdegen}(II) and \ref{twkar} below avoid the corresponding weights; in particular, they are functorially determined by the corresponding object $M$ (see Theorem \ref{tprkw}(\ref{iwildef1}) and Remark \ref{rwild}(\ref{irwild1})). 
\end{rema}

Now we are able to improve the ("partial") conservativity property of weight complexes given by Theorem 3.3.1(V) of \cite{bws}. %However, this requires certain definitions.

\begin{defi}\label{dwdegen}
%1. We will call the elements of $\cap_{i\in \z}\cu_{w\le i}$ (resp. of  $\cap_{i\in \z}\cu_{w\ge i}$) {\it  right degenerate}  (resp. {\it left degenerate}).

%2. We will say that $w$ is {\it non-degenerate} if $\cap_{i\in \z}\cu_{w\le i}=\cap_{i\in \z}\cu_{w\ge i}=\ns$.
%3. 
We will say that $M\in \obj \cu$ is  {\it $w$-degenerate} (or weight-degenerate) if  $t(M)$ is zero (in $\kw(\hw)$ and so also in $K(\hw)$). %$ 0$.
\end{defi}

\begin{theo}\label{tdegen}
Let $g:M\to M'$ be a $\cu$-morphism, $n\in \z$. Then the following statements are valid.

I.1. $t(g)$ is an isomorphism if and only if %$t(\co(g))=0$. % 
 $\co(g)$ is a $w$-degenerate object.

2. Any extension of a left   $w$-degenerate  object of $\cu$ by a  right $w$-degenerate one is $w$-degenerate.
 %If $M$ is a weight-degenerate object then $t(M)=0$.

3. If $M$ is %essentially $w$-negative (resp. essentially $w$-positive) 
an extension of a  left $w$-degenerate  object by an element of $\cu_{w\le n}$ (resp. is an extension of an element of $\cu_{w\ge n}$ by a right $w$-degenerate  object) 
then  $t(M)\in K(\hw)_{\wstu\le n}$ (resp. $t(M)\in K(\hw)_{\wstu\ge n}$; see Remark \ref{rstws}(1)).
%$t(M)\backsim^w 0$  (resp. $t(M)\backsim_w 0$; see Remark \ref{rwc}(\ref{irwc6})). 
%is equivalent to a complex concentrated in non-negative degrees (resp. in non-positive degrees).

II. Assume that  $\cu$ is weight-Karoubian.

 1. Then %the statements converse to assertions I.2 and I.3 are also valid. Being more precise,     $M$ is $w$-degenerate if and only  $t(M)=0$ 
 $M$ is $w$-degenerate if and only if $M$  is an extension of a left $w$-degenerate  object by a right $w$-degenerate one (cf. assertion I.2). % and 
 %$t(M)\backsim^w 0$  (resp. $t(M)\backsim_w 0$)

 2. $t(M)\in K(\hw)_{\wstu\le n}$ (resp. $t(M)\in K(\hw)_{\wstu\ge n}$)  if and only if  $M$ is an extension of a  left $w$-degenerate   object by an element of $\cu_{w\le n}$ (resp. is an extension of an element of $\cu_{w\ge n}$ by a  right $w$-degenerate   object; cf. assertion I.3).
\end{theo}
\begin{proof}
I.1. Immediate from Proposition \ref{pwt}(\ref{iwcex}) combined with the conservativity of the projection functor $K(\bu)\to \kw(\bu)$ (see Lemma \ref{lwwh}(\ref{iwhecat})).

2. If $N$ is  left or right $w$-degenerate   then  $t(N)=0$ according to Proposition \ref{pwt}(\ref{iwcons}). Hence the assertion follows from the previous one.

3. %Immediate from 
%The proof is a similar combination of assertion I.1  with  Proposition \ref{pwt}(\\ref{iwcons}).
Similarly to the previous assertion, it suffices to combine Proposition \ref{pwt}(\ref{iwcons}) with assertion I.1.

II. We investigate %the case of 
when $t(M)\in K(\hw)_{\wstu \le n}$. %??\backsim^w 0$. 

%This assumption on $M$ implies that 
 For any $m> n$ we have $\id_{t(M)}\backsim_{[-m,-n-1]}0$ (see  Lemma \ref{lwwh}(\ref{irwc6})). 
 
Since $\cu$ is weight-Karoubian, for any $n>0$ %there exists
Theorem \ref{tprkw}(\ref{iwildef2}) gives  a distinguished triangle $X_n\to M \to Y_m\to X_m[1]$
with $X_n\in \cu_{w\le n}$ and $Y_m\in \cu_{w\ge m+1}$. All of these triangles are isomorphic to the one for $m=n+1$ by the uniqueness statement in Theorem \ref{tprkw}(\ref{iwildef1}). Hence $Y_{n+1}$ is left $w$-degenerate and we obtain a triangle of the sort desired.

The proofs %in the two remaining cases
 of the two remaining statements are similar and left to the reader.
\end{proof}

\begin{rema}\label{rcwkar} 
1. Thus we get a precise answer to the question when $t(g)$ is an isomorphism in the weight-Karoubian case. In particular, the weight complex functor is conservative if  and only if $w$ is non-degenerate.

 %This statement is a significant improvement of Theorem 3.3.1(V) of \cite{bws} (that states that the restrictions of $t$ to the subcategories of left and right bounded objects of $\cu$ are conservative under this assumption).

2. To obtain the latter statement in the general case one should combine part I.1 of our theorem with %Corollary \ref{cwkar} 
 Theorem \ref{twkar} below.  %In  Theorem \ref{twkar} below %Corollary \ref{cwkar}  %we will %establish
%prove that part II of our proposition is also true "up to retracts" in any (not necessarily Karoubian) $\cu$.  
Moreover, that theorem contains several equivalent conditions for %$t(M)\backsim^w 0$, $t(M)\backsim_w 0$,
$t(M)\in K(\hw)_{\wstu\le 0}$ and $t(M)\in K(\hw)_{\wstu\ge 0}$ %, and $t(M)=0$ 
(for $\cu$ %being 
 that is not necessarily weight-Karoubian). % will be formulated. 
 However, those more general formulations %are more clumsy will 
 require the somewhat clumsy Definition \ref{dpkar}(4); we demonstrate that the corresponding modifications of Theorem \ref{tdegen}(II) are %necessary
	 unavoidable in %\S\ref{sindwd}%{del}
\S\ref{snkarex} below. 

% {dewpn} 
 %(and we will %generalize  prove that the corresponding modifications of Theorem \ref{tdegen}(II) are necessary in \S\ref{sindwd} below). 

3. Recall that $\cu$ is Karoubian whenever it is closed with respect to countable coproducts (triangulated categories satisfying this condition are called {\it countably smashing} in \cite{bsnew}) according to  Proposition 1.6.8 of \cite{neebook}. Hence it is quite reasonable to assume that $\cu$ is weight-Karoubian (see Proposition \ref{pbw}(\ref{ikwkar})). 
 %this property is rather reasonable to assume (at least) in the case where $\cu$ is "large".

4. Recall moreover that in the case where both $\cu$ and $\cu_{w\ge 0}$ are closed with respect to countable $\cu$-coproducts, $w$ is said to be {\it countably smashing} itself (in \cite{bsnew}; cf. Definition \ref{dbrown}(2) below). In this case for any $M\in \obj \cu$ there exists a essentially unique distinguished triangle $LM\to M\to RM\to LM[1]$ such that $RM$ is left weight-degenerate and $LM$ is left orthogonal to all    left weight-degenerate objects; see Theorem 4.1.3(1) of ibid. Thus this triangle coincides with the one provided by Theorem \ref{tdegen}(II) under the assumption $t(M)\in K(\hw)_{\wstu\le n}$. %in the case??!
\end{rema}

\subsection{On the relation to pure functors}\label{spure} %here or in \S4?!
We have just proved that the assumption that a $\cu$-morphism $g$ kills a given weight $m$ can be expressed in terms of $t(g)$. Now we use this statement to prove that this condition can be "detected" using pure functors (as defined in \S2.1 of \cite{bwcp}; %we will justify this terminology in Remark \ref{rchow}(2) 
this terminology was justified in Remark 2.1.3(3) of loc. cit.; cf. also Remark \ref{ray}(\ref{iray2}) below).  So, we recall some of the theory developed in ibid.

\begin{defi}\label{dpure}
Assume that  $\cu$ is endowed with a  weight structure $w$.

We will say that a (co)homological functor $H$ from $\cu$ into an abelian category $\au$ is {\it $w$-pure} (or just pure if the choice of $w$ is clear) if $H$ kills both $\cu_{w\ge 1}$ and $\cu_{w\le -1}$.
\end{defi}

%In the current paper we will prefer cohomological pure functors???!

\begin{theo}\label{tpure}
1. Let $\ca:\hw\to \au$ be an additive functor, where $\au$ is any abelian category. For an object $M$ of $\cu$ we will write  $t(M)=(M^j)$; we  set $H(M)=H^{\ca}(M)$ to be the zeroth homology of the complex $(\ca(M^{j}))$. Then $H(-)$ yields a homological functor. %$H:\cu\to \au$  that does not depend on the choices of weight complexes. 
  Moreover, the assignment $\ca\mapsto H^\ca$ is natural in $\ca$. 

2. The correspondence $\ca\to H^\ca$ is an equivalence of categories between the following (not necessarily locally small) categories of functors: $\adfu(\hw,\au)$ and the category of pure homological functors from $\cu$ into $\au$.

3. Dually, the correspondence sending a contravariant functor $\ca'$ into the functor $H_{\ca'}$ that maps $M$ into the  zeroth homology of the complex $(\ca'(M^{-j}))$ (see assertion 1) gives an equivalence of categories between $\adfu(\hw\opp,\au)$ and the category of pure cohomological functors from $\cu$ into $\au$.

4. A representable functor $\cu(-,M)$ is pure if and only if $M\in (\cu_{w\ge 1}\cup \cu_{w\le -1})\perpp$.

\end{theo}
\begin{proof}
Assertions 1 and 2 are contained in Theorem 2.1.2 of \cite{bwcp}, and assertion 3 is their dual (cf. Proposition \ref{pbw}(\ref{idual}) or Remark 2.1.3(1) of ibid.). Lastly, assertion 4 is immediate from the definition of purity. 
\end{proof}

Now we will prove that killing a given weight can be "detected" by means of pure functors. To prove that in certain cases representable functors are sufficient for this purpose, we recall the following definition.

\begin{defi}\label{dbrown}
Assume that  $\cu$ is closed with respect to small coproducts.

1. We will say that $\cu$ satisfies the  {\it Brown representability} property if  any cohomological functor from $\cu$ into $\ab$ that %respects ($\cu\opp$)-coproducts 
converts $\cu$-coproducts into products of groups is representable in $\cu$.

2. We will say that a weight structure $w$ on $\cu$ is {\it  smashing} if %$\cu$ is $\al$-smashing (resp., smashing)  and 
the class $\cu_{w\ge 0}$ is closed with respect to (small) $\cu$-coproducts %(as well; 
 (cf. Proposition \ref{pbw}(\ref{icoprod})). 
\end{defi}

\begin{pr}\label{puredetkw}
Assume that $\cu$ is endowed with a weight structure $w$, $g:M\to N$ is a $\cu$-morphism, and $j\in \z$.

I.  Then the following conditions are equivalent.

1. $g$ kills weight $j$. %See Definition \ref{dkw}?

2. $H^j_{\ca}(g)=0$ for any pure cohomological functor $H_{\ca}$ as above.

3.  $H^j_{\ca}(g)=0$ for any pure cohomological functor $H_{\ca}$ corresponding to a contravariant additive functor $\ca$ from $\hw$ into $\ab$ that converts all small $\hw$-coproducts into products of groups.

4. $H_j^{\ca}(g)=0$ for any pure homological functor $H^{\ca}$.

II. Assume in addition that $\cu$ satisfies the  Brown representability property and $w$ is smashing. Then the functors $H_{\ca}$  as in condition I.3 are precisely the pure representable ones.
\end{pr}
\begin{proof}
I. %Obviously, 
%We can assume that $j=0$ (see  Remark \ref{rwc}(\ref{irwcsh})). %easily seen?! Make this a part of {pwt}?!
%According to 
 It is easily seen that Proposition \ref{pwt}(\ref{irwcsh}) %, we can 
 allows us to assume that $j=0$.

Next, condition 1 implies condition 2 according to Proposition \ref{pwt}(\ref{iwhefu}) combined with %Proposition 
 Theorem \ref{tpwckill}(\ref{iwckill1}). Moreover, condition 2 obviously implies condition 3.

%It remains to 
Now let us prove that condition 3 (in the case $j=0$) implies that $g$ kills weight $0$. Thus we should verify the following: if $g$ does not kill weight $0$ then there exists $\ca$ as in assertion 3 such that $H_{\ca}(g)\neq 0$. %Moreover, Proposition \ref{pwckill}(\ref{iwckill1}) allows us to pass to weight complexes??
%Let us fix a choice $(N^i)$ of a weight complex for $N$; 
 For $t(N)=(N^i)$ and $d_N^i$ being the boundary morphisms of this complex we take the functor $\ca$ that sends any $X\in \cu_{w=0}$ into $\cok(\hw(X,-)(d^{-1}_N)):\hw(X,N^{-1})\to \hw(X,N^0)$. Obviously, this $\ca$ does convert $\hw$-coproducts into products in $\ab$. Hence it suffices to check that $H_{\ca}(g)\neq 0$.   

Now, $H_{\ca}(N)$ contains the element $\theta$ corresponding to $\id_{N^0}$. % (note that $d^0_N\circ d^{-1}_N=0$). 
 If $H_{\ca}(g)(\eta)=0$ then for %any choice $(M^i)$ of 
 $t(M)=(M^i)$ and $t(g)=(g^i)$ the element $\theta_M^0\in \ca(M^0)$ corresponding to %any choice of $g^0$ (here $t(g)=(g^i)$)
  $g^0$ belongs to the image of $\ca(M^1)$ (in $\ca(M^0)$).  This obviously implies that $t(g)\backsim_{[0,0]} 0$ %$t(g)\backsim_{[-m,-m]} 0$ %one summand only; is this possible??
(in the notation of Definition \ref{dbacksim}); thus it remains to apply %Proposition 
 Theorem \ref{tpwckill}(\ref{iwckill1})  once again.

Lastly, the equivalence of conditions I.1 and I.4 is just the categorical dual of the equivalence between I.1 and I.2.

II. This is Proposition 2.3.2(8) of \cite{bwcp}.
\end{proof}

\begin{rema}\label{rdetkw}
%1. Certainly, the categorical dual of this statement is also true, i.e., killing any weight can be detected by means of functors of the type $H^{\ca}$, where $\ca:\hw\to \ab$ is a functor respecting products. However, cohomological "detectors" of killing weights are more actual for our purposes below.
%2. This dual formulation 
 The equivalence of conditions I.1 and I.3 is closely related to Theorem 2.1 of \cite{barrabs}; the proof is similar is well. % (and its proof). 
\end{rema}

\section{On objects without weights in %non-weight-Karoubian
 general weighted categories, and a conservativity application} \label{snkar} 

In \S\ref{spkar}  we extend Theorem %s \ref{tprkw}(\ref{iwildef2}) and
  \ref{tdegen}(II) to the case where $\cu$ is not necessarily %just
	 weight-Karoubian. %These statements also enable us to generalize the latter theorem to arbitrary weight structures.  We also relate our results to intersections of "purely generated" subcategories. % (that are the main subject of \cite{binters}; cf. also \S4.3 of \cite{bsnew}).

In \S\ref{sprtwcons} we %prove that our results imply 
 apply our results to prove that certain weight-exact functors are "conservative up to weight-degenerate objects"; we also discuss the relation of this proposition to the corresponding results of \cite{wildcons} and \cite{bwcp}, and describe an interesting motivic example for it. %in contrast to??

In \S\ref{snkarex} we construct certain counterexamples to demonstrate that the modifications made in \S\ref{spkar} to %"adjust" the aforementioned results of the previous section 
 to generalize Theorem  \ref{tdegen}(II) cannot be avoided.

\subsection{On %killing weights in %weight-Karoubian extensions and generalizations of our results to 
 objects avoiding weights in not necessarily weight-Karoubian categories}\label{spkar}

To extend the results of the previous section to the case  of  a not (necessarily) weight-Karoubian $\cu$ we %need the following %We 
recall some definitions and results from \cite{bonspkar} (yet Definition \ref{dpkar}(4) is new).
%the central 
 %following definitions from \cite{bonspkar}.

\begin{defi}\label{dpkar}
1. We will call a triangulated category $\cu'$ an {\it  idempotent extension} of $\cu$ if  it contains $\cu$ and there exists a fully faithful exact functor  $\cu'\to \kar(\cu)$.\footnote{Recall that (according to Theorem 1.5 of \cite{bashli}) the category  $\kar(\cu)$ can be naturally endowed with the structure of a triangulated category so that the  natural embedding functor $\cu\to \kar(\cu)$ is exact. Hence
%The latter assumption is certainly equivalent 
$\cu'$ is an idempotent extension of $\cu$ if  and only if any object of $\cu'$ is a retract of some object of $\cu$.} %???and $\cu$ is dense (see \S\ref{snotata})  in $\cu'$.}.

2. We will say that a weight structure $w$  {\it extends} to an  idempotent extension $\cu'$ of $\cu$ %is {\it $w$-compatible} with $\cu$ 
 whenever there  exists a weight structure %$w$ for $\cu$ and
 $w'$ for $\cu'$ such that the embedding $\cu\to \cu'$ is weight-exact.  In this case we will call $w'$ an {\it extension} of $w$. %the?! cf. ? below?!

3. We will call  a weight-Karoubian category $(\cu',w')$ (see Definition \ref{dwso}(\ref{idwkar})) a   {\it weight-Karoubian extension} of $(\cu,w)$ if $\cu'$ is an  idempotent extension of $\cu$ and $w'$ is the extension of $w$ to it (cf. Proposition \ref{ppkar}(1)).

4.  We will say that %For 
 an object $M$ of $\cu$ % (where $\cu$ is endowed with a weight structure $w$) we will say that $M$ 
 is {\it essentially $w$-positive}  (resp. {\it essentially $w$-negative}) if it is a retract of some $M'\in \obj \cu$ such that  $M'$ is an extension of an element of $\cu_{w\ge 0}$ by a right $w$-degenerate object of $\cu$  (resp. $M'$ is an extension of  a left $w$-degenerate object of $\cu$ by an element of $\cu_{w\le 0}$; see Definition \ref{dwso}(\ref{ilrd})). 
\end{defi}

%Now we recall those results of ibid. that are relevant for the current paper.

\begin{pr}\label{ppkar}
1. Let $\cu'$ be an  idempotent extension of $\cu$ such that  $w$  extends to a weight structure $w'$ on it. Then  $\cu'_{w\ge 0}$ (resp.  $\cu'_{w'\le 0}$, resp.  $\cu'_{w'= 0}$)  is the retraction-closure of  $\cu_{w\ge 0}$ (resp.  $\cu_{w\le 0}$, resp.  $\cu_{w= 0}$) in $\cu'$.  In particular, $w$ is the restriction of $w'$ to $\cu$.

2. Any $(\cu,w)$ possesses a weight-Karoubian extension.

\end{pr}
\begin{proof}
1. This is Theorem 2.2.2(I.1) of ibid.

2. The statement is given by part III.1 of loc. cit.
\end{proof}

%The following observations are crucial for this section.

Now we %use Proposition \ref{pwkar} for deducing
 %prove a certain version of 
 extend Theorem  \ref{tdegen}(II) %that would be valid for a
 to  not (necessarily) weight-Karoubian weighted categories.

\begin{theo}\label{twkar}
Let $M\in \obj \cu$. 

I. %Then the following statements are valid.
The following conditions are equivalent.
\begin{enumerate}
\item\label{icwk1} %$M$ is weight-degenerate 
%$t(M)=0$ 
$M$ is weight-degenerate (resp. $t(M)\in K(\hw)_{\wstu\le 0}$). %is homotopy equivalent to a complex concentrated in non-negative degrees).

\item\label{icwk2}  $M$ can be presented as an extension of a left $w'$-degenerate  object of $\cu$ by a  right $w'$-degenerate   one (resp. by an element of $\cu'_{w'\le 0}$) in some weight-Karoubian extension $(\cu',w')$ of $(\cu,w)$.

\item\label{icwk3} Such a decomposition of $M$ exists in any weight-Karoubian extension of $\cu$.

\item\label{icwk4} $M$ is a $\cu$-retract of an extension $M'$ of a left $w$-degenerate   object of $\cu$ by a  right weight-degenerate   one (resp.  $M$ is essentially $w$-negative in the sense of  Definition \ref{dpkar}(4)).  %Definition \ref{dewpn}.
%a $\cu$-retract of an extension of a left $w$-degenerate   object of $\cu$ by a  right weight-degenerate   one (resp. by  an element of $\cu_{w\le 0}$, i.e., $M$ is essentially $w$-negative in the sense of Definition \ref{ddeg}(\ref{iepn})).

\item\label{icwk5} The object $M\bigoplus M[-1]$ is an extension %of this sort. % of a degenerate above object 
 a left $w$-degenerate  object of  $\cu$ by a  right $w$-degenerate   one (resp. by  an element of $\cu_{w\le 0}$).

\item\label{icwkw} $M$ is without weight $i$ for all $i\in \z$ (resp. for all $i>0$).

\item\label{icwpuh} $H_i(M)=0$ for all $i\in \z$  (resp. for all $i>0$) and pure homological $H$. %homological? resp?!

\item\label{icwpuc} $H_{\ca}^i(M)=\ns$ (see Theorem \ref{tpure}(3)  for the notation)  for all $i\in \z$  (resp. for all $i>0$) and all additive functors $\ca:\hw\opp\to \ab$ that respect products. 

\end{enumerate}

II. The following conditions are equivalent as well.

\begin{enumerate}
\item\label{icwk6} $t(M)\in K(\hw)_{\wstu\ge 0}$. %\backsim_w 0$. %$  is  homotopy equivalent to a complex concentrated in non-positive  degrees.

%\item\label{icwk6r} $t(M)$ is a retract of a complex  concentrated in non-positive  degrees (in $K(\hw)$). 

\item\label{icwkwp} $M$ is without weight $i$ for all  $i<0$.

\item\label{icwk7}  $M$ can be presented as an extension of an element of $\cu'_{w'\ge 0}$ by a  right weight-degenerate  object   in some weight-Karoubian extension $\cu'$ of $\cu$.

\item\label{icwk8} Such a decomposition of $M$ exists in any weight-Karoubian extension of $\cu$.

\item\label{icwk9} $M$ is essentially $w$-positive in the sense of Definition \ref{dpkar}(4). %Definition \ref{dewpn}. %\ref{ddeg}(\ref{iepn}).  %a $\cu$-retract of an extension %of  a degenerate above object of $\cu$ by a degenerate  below one
%(resp. of 
%of an element of $\cu_{w\ge 0}$ by a  right degenerate  object.

\item\label{icwk0} The object $M\bigoplus M[1]$ is an extension of an element of $\cu_{w\ge 0}$ by a  right $w$-degenerate  object. %this sort.

\item\label{icwpuhp} $H_i(M)=0$ for all $i<0$ and pure homological $H$. %homological? resp?!

\item\label{icwpucp} $H_{\ca}^i(M)=\ns$   for all $i<0$ and all additive functors $\ca:\hw\opp\to \ab$ that respect products. 
\end{enumerate}
\end{theo}
\begin{proof} We will only prove assertion II; the proof of assertion I is similar. %????

Clearly, condition \ref{icwk0} of the assertion implies condition \ref{icwk9}. %, and 
\ref{icwk7} follows from \ref{icwk8} since a weight-Karoubian extension $(\cu',w')$ of $\cu$ exists (see Proposition \ref{ppkar}(2)).

Condition \ref{icwk6} is easily seen to be equivalent to condition \ref{icwkwp} according to %Proposition 
 Theorem \ref{tpwckill}(\ref{iwckill3},\ref{iwckill1}) combined with Lemma \ref{lwwh}(\ref{irwc3},\ref{irwc6}). Moreover, condition \ref{icwk6} is equivalent to $t(M)\backsim_{[i,i]}0$ for all $i>0$ according to Lemma \ref{lwwh}(\ref{irwc3},\ref{irwc6}); thus applying Proposition \ref{puredetkw}(I) we obtain that this condition is also equivalent to conditions \ref{icwpuhp} and \ref{icwpucp}.

Next we note that (for any  weight-Karoubian extension $\cu'$ of $\cu$ and a fixed $M$)
%we have 
%$t(M)\backsim_w 0$ in $K(\hw)$ if and only if this is true in $K(\hw')$
$t(M)\in K(\hw)_{\wstu\ge 0}$ if and only if it belongs to $K(\hw')_{\wstu\ge 0}$; see   Lemma \ref{lwwh}(\ref{iwhefun},\ref{irwcwd},\ref{irwc6}). Hence  condition \ref{icwk9} implies condition \ref{icwk6}.
%Moreover, \ref{icwk6} is equivalent to \ref{icwk6r} by Remark \ref{rwc}(\ref{irwc6}).

Now we fix some $(\cu',w')$  and recall that (the conclusion of)
Theorem  \ref{tdegen}(II.2) %extends to weight-Karoubian categories 
can be applied to $\cu'$. %???? according to  Proposition \ref{pwkar}(1). 
 Hence  condition \ref{icwk6}  implies condition \ref{icwk8}. %Moreover, 
%Next, $M$ is  without weight $i$ for all  $i<0$ in $\cu$ if and only if it is so in $\cu'$ according to Theorem \ref{tprkw}(\ref{iprkwfunct},\ref{iprkwfunctemb}).

It remains to deduce condition \ref{icwk0} from condition \ref{icwk7}. 
Any $N'\in \obj \cu'$ %that 
is  the %"formal image"{del}
 image of an idempotent $p\in \cu(N,N)$ for some $N\in \obj \cu$ (see \S\ref{snotata}), and $\co(p)\cong N'\bigoplus N'[1]\in \obj \cu$ (cf. Lemma 2.2 of \cite{thom}). Hence the direct sum of the $\cu'$-"decomposition" of $N$ given by condition \ref{icwk7} with its shift by $[1]$ yields condition  \ref{icwk0}.
\end{proof}

\begin{rema}\label%{rptwcons} %
	{rwkarloc}
1. Clearly, the formulation of conditions  I.\ref{icwpuc} and II.\ref{icwpucp} can be combined with Proposition \ref{puredetkw}(II) to obtain the following statement: if $\cu$ satisfies the  Brown representability property and $w$ is smashing then an object $M$ is $w$-degenerate (resp. essentially $w$-negative, resp. essentially $w$-positive) if and only if $H^i(M)=\ns$ for any pure representable $H$ and any $i\in \z$ (resp. any $i>0$, resp. any $i<0$). 
	%$M\perp \cup_{i\in \z}\obj \hrt[i]$ (resp. $M\perp \cup_{i\in \z}\obj \hrt[i]$, ; 
	
	Note also that conditions   I.\ref{icwpuh},  I.\ref{icwpuc}, II.\ref{icwpuhp}, and II.\ref{icwpucp} of our theorem can be dualized in the obvious way. 
	%smashing Brown?!
	
	2. Recall that   \cite{bonspkar} %much more
	 contains %some extra 
	 much information on  idempotent extensions of $\cu$ such that $w$ extends to them. % is contained. 
	 In particular, the  (essentially) minimal weight-Karoubian extension $\cu'$ of $\cu$ was described as the smallest %Karoubi-closed 
 (strict) triangulated subcategory of $\kar(\cu)$ that contains both $\cu$ and $\kar(\hw)$. 
%$\wkar(\cu)=\lan \obj \kar(\cu^-)\cup \obj \kar(\cu^+) \ra_{\kar(\cu)}$. 
Since it is minimal, applying %our proposition
  conditions %I.\ref{icwk2} and II.\ref{icwk7}
	 I.\ref{icwk3} and II.\ref{icwk8} of Theorem \ref{twkar} in this $\cu'$  gives the maximal possible amount of information on $M$. % the corresponding $\cu$.

%It was also demonstrated (in
 3. Moreover, in \S3.1 of ibid. it was  demonstrated that a weight structure $w$ does not necessarily extend to the (whole) %Karoubi envelope of 
 category $\kar(\cu)$; thus idempotent extensions are necessary for our arguments.
\end{rema}

	Now let us prove a few results closely related to our theorem.
	
	\begin{pr}\label{prwkarloc}
	
	Let  $w$ be a weight structure on $\cu$.	
	
%Then the following statements are valid.
1. Assume that $w$ is left (resp. right) non-degenerate.

Then any weight-degenerate object of $\cu$ is right (resp. left) $w$-degenerate, and any %its 
 essentially $w$-negative (resp. essentially $w$-positive) object belongs to $\cu_{w\le 0}$ (resp. to $\cu_{w\ge 0}$).

2. Assume that $w$ is left  non-degenerate,  an object $M$ of $\cu$ is weight-degenerate, and either $M$  is  $w$-bounded below or $w$ is also right non-degenerate. Then $M$ is zero.

3. For an object $M$ of $\cu$ assume that $t(M)\cong t_1=(M_1^i)\cong t_2=(M_2^i)$ (in the categories $\kw(\hw)$ and $K(\hw)$; see Lemma \ref{lwwh}(\ref{iwhecat})). Let $\cu'$ be a triangulated category endowed with a weight structure $w'$; let $F:\cu\to \cu'$ be a weight-exact functor (with respect to $w,w'$), and assume that $F$ annihilates the groups $\cu(M_1^i,M_2^i)$ for all $i\in \z$.

Then $F(M)$ is $w'$-degenerate. 

4. Assume that $\cu_1,\ \cu_2$, and $\cu_3$ are full triangulated subcategories of $\cu$ such that $w$ restricts to them  (see Definition \ref{dwso}(\ref{idrest})), and suppose that the Verdier localization functor $F:\cu\to \cu'=\cu/\cu_3$ exists (i.e., all morphism classes in this localization are sets) and   all $\hw$-morphisms between elements of the corresponding classes $\cu_{1,w_1=0}$ and $\cu_{2,w_2=0}$ are killed by $F$. %the localization functor  

Then there exists a weight structure $w'$ on $\cu'$ such that $F$ is weight-exact, and for any $M\in \obj \cu_1\cap \obj\cu_2$ the object $F(M)$ is weight-degenerate. Moreover, if $w'$ is left non-degenerate then $F(M)$ is $w'$-right degenerate, and if we assume in addition that $M$ is $w$-bounded below then $M$ belongs to $\kar_{\cu}(\obj \cu_3)$.
\end{pr}
\begin{proof}
1. This is an easy consequence of the axiom (i) of Definition \ref{dwstr} along with the following fact %the following observation: 
 that follows from this axiom immediately:  retracts of left and right $w$-degenerate objects are left and right $w$-degenerate, respectively. %immediately from the .

2.  According to the previous assertion, $M$ is right weight-degenerate, i.e., it belongs to $\cap_{i\in \z}\cu_{w\le i}$. On the other hand, $M$ belongs to $\cu_{w\ge i+1}$ for some $i\in \z$ (and it actually belongs to all of these classes in the second case by the previous assertion). Since $\cu_{w\le i}\perp \cu_{w\ge i+1}$, $M\perp M$; hence $M=0$.

3. Let $m$ be a $\kw(\hw)$-isomorphism $t_1\to t_2$. Then for the functor $\kw(\hf):\kw(\hw)\to \kw(\hw')$ given by  Lemma \ref{lwwh}(\ref{iwhefun}) (here $\hf$ is the restriction of $F$ to hearts) we obviously have $\kw(\hf)(m)=0$. Since $\kw(\hf)(m)$ is also an isomorphism, we obtain $\kw(\hf)(t_1)=0$. On the other hand, by Proposition \ref{pwt}(\ref{iwcfunct}) we have $t_{w'}(F(M))\cong \kw(\hf)(t(M))\cong \kw(\hf)(t_1)$; hence $t_{w'}(F(M))=0$, as desired. %by definition???%in particular??! A separate part?!

4. $w'$ exists according to  Proposition 8.1.1(1) of  \cite{bws}.  Next, Proposition \ref{pwt}(\ref{iwcfunct}) gives the existence of complexes $t_1$ and $t_2$ in $\kw(\hw)$  such that $t(M)\cong t_1\cong t_2$, the terms $M_1^i$ belong to $\cu_{1,w_1=0}$, and the terms of $t_2$ belong to $\cu_{2,w_2=0}$.  Thus applying assertion 3 we  obtain that $F(M)$ is $w'$-degenerate indeed. 

By assertion 1 it follows that $F(M)$ is right $w'$-degenerate whenever $w'$ is left non-degenerate. Lastly, if $M$ is weight-bounded below then $F(M)$ also is. Hence our assumptions imply that $F(M)=0$ according to assertion 2. Thus $M$ belongs to $\kar_{\cu}(\obj \cu_3)$ (here we apply a well-known property of Verdier localizations; see Lemma  2.1.33 of \cite{neebook}).
\end{proof}

	\begin{rema}\label{rbinters}
	Now let us describe an application of this proposition that is crucial for \cite{binters};\footnote{Since \cite{binters} was written earlier than the current paper, it actually refers to a previous version of this text. However, the exposition in the current version is more accurate. Note also that the most recent version of the theory of weight complexes (that was applied in ibid. as well) is currently contained in \cite{bwcp}.} cf. also \S4.3 of \cite{bsnew}.
	
	Assume that $w$ is the weight-structure {\it purely compactly generated} by a subcategory $\hu$ of $\cu$ in the sense of Theorem \ref{thegcomp} below (in particular, $\cu$ and $\hu$ satisfy the assumptions of the theorem), and assume that $\cu_1$, $\cu_2$, and $\cu_3$  are %compactly 
	 generated %(see Definition \ref{dcomp} below) %(\ref{idcompg}))
		by the corresponding subcategories $\hu_1$, $\hu_2$, and $\hu_3$ of $\hu$ as localizing subcategories of $\cu$ (see Definition \ref{dcomp}(\ref{idlocal}) below), respectively. As demonstrated in (Proposition 1.8 and the proof of Proposition 1.9 of) ibid., if all morphisms between $\hu_1$ and $\hu_2$ factor through $\hu_3$ then all the assumptions of  Proposition \ref{prwkarloc}(4) are fulfilled, and the corresponding weight structure $w'$ is left non-degenerate.\footnote{Actually, in ibid. the subcategories $\hu,\hu_1,\hu_2$, and $\hu_3$ were assumed to be additive; yet this assumption is easily seen not to be actual for the purposes of that paper. Note also that the only part of the proof of   \cite[Proposition 1.9]{binters} that we apply in this remark is that the "factorization condition" for the categories $\hu_i$ ($i=1,2,3$) implies the one for the classes $\cu_{i,w_i=0}$, and this implication is an easy consequence of  \cite[Proposition 1.8]{binters} (cf. also Theorem \ref{thegcomp}(\ref{itnp1d}) below).} Moreover, $\cu_3$ is Karoubian (see Remark \ref{rcwkar}(3)); hence if $M$ is $w$-bounded below and belongs to $\obj \cu_1\cap \obj \cu_2$ then $M$ is an object of $\cu_3=\kar_{\cu}\cu_3$.
		
		Thus our  Proposition \ref{prwkarloc} gives all the necessary prerequisites for the proof of   Proposition 1.9 of ibid. (that is critically important for the whole paper).%loc. cit.
\end{rema}

\subsection{An application to the conservativity of weight-exact functors}\label{sprtwcons}

Our results imply that certain weight-exact functors are ("almost") conservative; note that this statement does not mention weight complexes. 

\begin{pr}\label{pwcons}
Let $\cu$ and $\cu'$ be  triangulated categories endowed with weight structures $w$ and $w'$, respectively; let $F:\cu\to \cu'$ be a weight-exact functor.

Assume that the induced functor $\hf:\hw\to \hw'$ is full, %and conservative, 
any $\hw$-endomorphism killed by $\hf$ is nilpotent, 
and  for some $M\in \obj \cu$ %its image $F(M)$ is zero
 the object $F(M)$ belongs to  $\cu'_{w'\notin[m,n]}$ %{dkw}(2) is without weights  
for some $m\le n\in \z$ (resp. $F(M)$ is $w'$-degenerate, resp. $F(M)$ is essentially $w'$-positive, resp. essentially $w'$-negative). %, resp. %$F(M)\in \cu'_{w'\le 0}$, resp.  ).

Then $M$ is  without weights $m,\dots, n$  (resp. $M$ is weight-degenerate, resp. essentially $w$-positive, resp. essentially $w$-negative). % resp. without weights $m,\dots, n$).
\end{pr}
\begin{proof}
We start from proving the first statement in our proposition for $m=n$, i.e., we assume that $F(M)$ is without weight $m$. 
We should prove that this assumption implies that $M$ is without weight $m$ as well; we will call this implication Claim (*).

According to %Proposition 
 Theorem \ref{tpwckill}(\ref{iwckill3}) this claim is equivalent to the following one: $t(\id_M)\backsim_{[-m,-m]}0$ whenever $t_{w'}(\id_{F(M)})\backsim_{[-m,-m]}0$. %implies that  if and only if $t(\id_M)\backsim_{[-m,-m]}0$. 

Now, we can assume that $t_{w'}(F(M))$ is obtained from %a choice of a 
 the weight complex $t(M)=(M^i)$ (whose boundary morphisms will be denoted by $d^i$) by means of termwise application of $F$; see Proposition \ref{pwt}(\ref{iwcfunct}). Thus there exist morphisms $h'\in \hw'(F(M^{-m}),F(M^{-m-1}))$  and $j'\in \hw'(F(M^{1-m}),F(M^{-m}))$ such that $\id_{F(M^{-m})}=j'\circ F(d^{-m}) +F(d^{-m-1})\circ h'$. Since the restriction of $F$ to $\hw$ is full and conservative, we can  lift $h'$ and $j'$ to some $\hw$-morphisms $h$ and $j$, and for any such lifts  the endomorphism  $\eps=%j\circ d^{-m} +d^{-m-1}\circ h-\id_{M^{-m}}$ 
\id_{M^{-m}}-j\circ d^{-m} -d^{-m-1}\circ h$  of $M^{-m}$ is nilpotent. %an automorphism of    a  nilpotent automorphism . 
 Hence there exists $n>0$ such that $\id_{M^{-m}}=(j\circ d^{-m} +d^{-m-1}\circ h)(\id_{M^{-m}}+\eps+\eps^2+\dots+ \eps^{n-1})$. %and it easily follows that 
 Therefore $\id_{M^{-m}}$ can be presented in the form $a\circ d^{-m}+  d^{-m-1}\circ b$ for some $a\in \hw(M^{-m},M^{-m-1})$  and $b\in \hw(M^{1-m},M^{-m})$ (%since %there exist
one can just write down explicit formulas for $a$ and $b$ in this setting). Thus $t(\id_M)\backsim_{[-m,-m]}0$. 
	
	%$t(\id_M)\backsim_{[-m,-m]}0$ indeed. %explain; more detail?!!?!

Next, the general case of the "without weights $m,\dots, n$ part" follows from %the $m=n$-one 
 Claim (*) immediately according to Theorem \ref{tprkw}(\ref{iprkws}).

Lastly, our remaining statements  follow from Claim (*) as well according to  Theorem \ref{twkar}; see conditions I.\ref{icwkw} and II.\ref{icwkwp}  in it. 
\end{proof}

\begin{rema}\label{ray}
\begin{enumerate}
\item\label{iray1}
Our proposition essentially says that $F$ is  "conservative (and detects weights; cf. Remark 1.5.3(1) of \cite{bwcp}) up to weight-degenerate objects". The latter feature of the result is unavoidable. Indeed, arguing similarly to Proposition 4.2.1(1) of \cite{bsnew} one can easily prove that for any set $W$ of weight-degenerate objects of $\cu$ the localization of $\cu$ by  
 the triangulated subcategory generated by $W$ gives a weight-exact functor that restricts to a full embedding of $\hw$ (as well as of the subcategory $\cu^b$ of $w$-bounded objects) into $\cu'$. 

%Note also that
Another example is given by singular homology. It is well-known to give an exact functor $\shtop\to D(\ab)$ whose kernel consists of {\it acyclic spectra}; hence this kernel is non-zero (see Theorem 16.17 of \cite{marg}). Next,  this functor is weight-exact and gives an equivalence of hearts if we take $w$ to be the weight structure purely compactly generated by the sphere spectrum $S^0\in \obj \shtop$ and $w'$ to be   {\it purely compactly generated} by $\z\in \obj D(\ab)$; see Theorems \ref{thegcomp}(\ref{itnp1},\ref{itnp1d},\ref{itnwe}) and \ref{tsh}(\ref{itoptriv},\ref{itopsingh}) below.
%recall restrictions corresponding to subcategories??!

%generalization of   \cite[Theorem 2.8(d)]{wildcons}??!! semi-primary+full+conservative: easily seen to imply?!
\item\label{iray2} Now let us discuss certain motivic examples to our proposition.

Firstly, we recall that "any reasonable version" of the category $\dm$ of Voevodsky motives over a perfect field $k$ contains the corresponding connective subcategory $\chow$ (of Chow motives) whose objects are compact in $\dm$ (see Definition \ref{dcomp} below and  Remark 3.1.2 of \cite{bokum}). Moreover, the category $\chow$ compactly generates $\dm$ whenever the characteristic of $k$ is either zero or invertible in the coefficient ring of these categories. Applying Theorem \ref{thegcomp} below one obtains a certain weight structure $\wchow$ on $\dm$ whose heart consists of retracts of coproducts of Chow motives. Moreover, this weight structure restricts (see Definition \ref{dwso}(\ref{idrest})) to the subcategory $\dmgm$ of  geometric motives (this is the smallest  retraction-closed triangulated subcategory of $\dm$ containing $\chow$) as well as to the bigger subcategory $\dm^{\chow}$ of $\dm$ whose objects are those $M\in \obj \dm$ whose weight complex $t(M)$ is $\kw({\hw}_{\chow})$-isomorphic to an object of $\kw(\chow)$, and the heart of both of these restrictions is $\chow$; %English?? 
see Remarks 3.1.4(1) and 3.3.2(1) of \cite{bwcp} for more detail.

Let us now describe an important weight-exact motivic functor essentially originating from \cite{ayoubcon}. In ibid. the  case $\cha k=0$ and a certain triangulated category $\da$ of \'etale motives was considered; $\da$ is a $k$-linear version of $\dm$. Moreover, the  truncated de Rham spectrum $\tau_{\ge 0}\omdr$ (see Theorem II of ibid.) was taken; note that this object $\omp$ of $\da$  it is a highly structured ring spectrum with respect to the  model structure on $\da$ that was considered in ibid.  Thus one can take $\cu'$ to be the derived category of highly structured $\omp$-modules in $\da$; there is a natural "free module" functor $F'=-\otimes \omp:\da\to \cu'$. Moreover, the images of Chow motives in $\cu'$ yield a weight structure $w'$ on this category, the heart of $w'$ is the formal coproductive hull (see Definition \ref{dcomp}(\ref{idhull}) below)
 of the category of $k$-linear motives up to algebraic equivalence, and  %$-\otimes \omp$ 
 $F'$ is weight-exact with respect to these weight structures (see Remark 3.3.2(1) of \cite{bwcp}). Since an endomorphism of Chow motives that is algebraically equivalent to zero is nilpotent according to Corollary 3.3 of \cite{voevnilp}, the restriction of $F'$ to the subcategory $\dm^{\chow}=\da^{\chow}$ satisfies the assumptions of our proposition. 

\item\label{iray3} Now let us compare our proposition with similar results of \cite{wildcons} and \cite{bwcp}. 

In Theorem 1.5.1(1,2) of  \cite{bwcp} only the case where $M$ is bounded either above or below was considered. On the other hand, $\hf$ was %only
 just  assumed to be (full and) conservative. So, neither our proposition implies loc. cit. nor the converse is valid.\footnote{It is an interesting question whether it suffices to assume that $\hf$ is full and conservative in our proposition. } Note also that in the case where the endomorphisms killed by $\hf$ are nilpotent all the conclusions of loc. cit. can be easily deduced from our proposition. % (see of ibid.).

Next we recall that in Theorem 2.8 of \cite{wildcons} (as well as in the weaker Theorem 2.5 of  ibid.) it was assumed (if one uses our notation) that $F$ is weight-exact, $\hf$ is full and conservative, $w$ is {\it bounded} (i.e., any object of $\cu$ is $w$-bounded both above and below), $\hw$ is Karoubian and {\it semi-primary}. Now, these assumptions imply that endomorphisms killed by $\hf$ are nilpotent; indeed, the conservativity of $\hf$ means that these endomorphisms belong to the {\it radical} (morphism ideal) of $\hw$, and semi-primality means that that all elements of this radical are nilpotent. Thus our proposition implies  Theorem 2.8 %(d) 
 of ibid. (cf. Theorem 1.5.1(1) of \cite{bwcp}). %Moreover, it easily implies the remaining parts of loc. cit. (see Remark 1.5.3(2) of \cite{bwcp} for a similar argument that  yields this claim).
%Let us explain that these assumptions on $\hf$ imply our ones. Indeed,  since   $\hw$

\item\label{iray4} One can also obtain plenty of examples to our proposition by taking $F=K(G):K(\bu)\to K(\bu')$; %starting from a full additive functor
 here  $G:\bu\to \bu'$ is a full additive functor such that any $\bu$-endomorphism killed by $G$ is nilpotent, and one takes $w$ and $w'$ to be the corresponding stupid weight structures. Note that in this case we have $\hf\cong \kar(G)$ (see Remark \ref{rstws}(1)); thus $\hf$ fulfils our assumptions (as well). 

 Since $w$ is non-degenerate, our proposition implies that $F$ is conservative. The author wonders whether %any alternative 
 a proof of this statement "without  killing weights" exists; this %would possibly
  may  allow to modify the assumptions on $\hf$ in Proposition \ref{pwcons}.
%Other examples to our proposition are discussed in Remark \ref{rnd}(3). 
\end{enumerate}
\end{rema}

\subsection{Some counterexamples in the non-weight-Karoubian case}\label{snkarex} %section?! A,B,C?!!

Our examples %will be 
 are rather simple; their main "ingredient" is $K(\lvect)$ (the homotopy category of %bounded 
 complexes of finite dimensional $L$-vector spaces; here $L$ is an arbitrary fixed field).

\subsubsection{An "indecomposable" weight-degenerate object }\label{sindwd}
 
Let us demonstrate that Theorem  \ref{tdegen}(II) and %Proposition \ref{pwkar}(2) 
 Theorem \ref{tpwckill}(\ref{iwckill4})  do not extend to arbitrary (i.e., to not necessarily weight-Karoubian) triangulated categories; %thus 
 in particular, retracts cannot be avoided when one defines essentially $w$-negative and essentially $w$-negative objects.

Our example will be the full subcategory $\cu$ of $(K^b(\lvect))^3$ consisting of objects whose "total Euler characteristic"  is even (i.e., the sum of dimensions of all homology of all the three components of $M=(M_1,M_2,M_3)$ should be even). We define $w$ for $\cu$ as follows: $\cu_{w\le 0}$ consists of those $(M_1,M_2,M_3)$ such that $M_1\cong 0$ and $M_2$ is acyclic in negative degrees; %=0? 
$(M_1,M_2,M_3)\in \cu_{w\ge 0}$ if $M_3\cong 0$ and  $M_2$ is acyclic in positive degrees. This is easily seen to be a weight structure; indeed,   a weight decomposition of an object $(M_1,M_2,M_3)$ is given by any triangle of the form $(0,M',M_3)\to  (M_1,M_2,M_3)\to (M_1,M'',0)$, where %$M'\to M_2'\to M''$ is a stupid weight decomposition of a complex $M_2'$ that is isomorphic to  $M_2$ in $K^b(\lvect)$ %with  and
$M'\to M_2\to M''$  is a %stupid weight 
$\wstu$-decomposition of $M_2$ with  the corresponding parities of the Euler characteristics. Next, one can easily see that the object $M=(L,0,L)$ (here we put the $L$'s in degree $0$ though the degrees actually do not matter) % make no difference)
  is weight-degenerate (since it is weight-degenerate in the obvious extension of $w$ to its weight-Karoubian extension $\cu'=(K^b(\lvect))^3$; see Proposition \ref{pwt}(\ref{iwcfunct})). Yet $M$ clearly cannot be presented as an extension of a left $w$-degenerate  object (i.e., of an object whose last two components are zero) by  %a right degenerate one (whose first two components are zero). 
an element of $\cu_{w\le 0}$ (since the corresponding "decomposition" in $\cu'$ is unique and its "components" have odd "total Euler characteristics"). Thus we obtain that %first two  statements in 
 both part of Theorem  \ref{tdegen}(II) do not extend to $\cu$. %; for the same reasons, the third statement in loc. cit. does not extend to $\cu$ either (and the same $M$ does not possess the corresponding "decomposition"). 

Looking at the proof Theorem  \ref{tdegen}(II) one immediately obtains the existence of $n>0$ such that %for the object considered above 
there does not exist a triangle $X_n\to M \to Y_n$
with $X_n\in \cu_{w\le -n}$, $Y_n\in \cu_{w\ge n}$. Moreover, one can easily check directly that a triangle of this sort does not exist for $n=1$ already.

This example also demonstrates that one has to assume that $\cu$ is weight-Karoubian in %Proposition \ref{pwkar}(2).
 Theorem \ref{tpwckill}(\ref{iwckill4}).   Indeed, for our object $M$ we have $t(M)=0$ since $M$ is weight-degenerate. %; yet $M$ does not possess

%in the bounded case: use a $K_0$-argument?!

\subsubsection{A bounded object that  is  without  weight $0$ but does not possess a decomposition avoiding this weight}\label{ssskwild}

 The example above demonstrates that Theorem \ref{tprkw}(\ref{iwildef2}) does not extend to arbitrary $(\cu,w)$ (i.e., that our definition of objects without weights $m,\dots,n$ is not equivalent to Definition 1.10 of \cite{wild} in general). Yet the weight structure is degenerate in this example. Now we give a bounded %counterexample %of the latter sort.
%to the latter extension statement.
example of the non-equivalence of definitions (i.e., all objects of $\cu$ are bounded both above and below; this condition is easily seen to imply that $w$ is non-degenerate).

%\item\label{irwild5}
 %Assertion  \ref{iwildef2} of the proposition is wrong for a general $\cu$ (i.e., for a one that is not necessarily Karoubian). 
%As an example, we  For example, d
Denote by $\bu$  the category of even-dimensional vector spaces over %(any) fixed field 
$L$; take $\cu=K^b(\bu)$, $$M=\dots \to 0\to L^2\stackrel{\begin{pmatrix}
 1 & 0  \\
 0 & 0 %\\
\end{pmatrix}}{\longrightarrow} L^2\stackrel{\begin{pmatrix}
 0 & 0  \\
 0 & 1 %\\
\end{pmatrix}}{\longrightarrow} L^2 \to 0\to\dots;$$ we put the %se 
 non-zero vector spaces in degrees $-1,0$, and $1$, respectively. %complete by 0s: earlier?! put the first N in degree????!
Clearly, the composition $(0\to L^2\to L^2) \to M\to (L^2\to L^2\to 0)$ is zero %(these $L^2$ are in degrees )
(here we consider the corresponding stupid truncations of $M$). Thus  $M$ is without weight $0$ (see condition \ref{ikw1} in   Proposition \ref{pkillw}). Yet $M$ does not possess a  decomposition % (\ref{ewild}) for $m=n=0$ 
avoiding weight $0$ since the $L$-Euler characteristics of the corresponding $X$ and $Y$ cannot be odd.

Obviously, this example also yields that %decompositions of the type (\ref{ewild}) 
 decompositions avoiding weights  $m,\dots, n$ do not "lift" from a (weight-)Karoubian $\cu'$ (in our case $\cu'=K^b(\lvect)$; the corresponding weight structure is the stupid one) to $\cu$ (cf. Remark \ref{rwild}(\ref{irwild6})). % (that is endowed with $w$ and contains $M$). 

\section{On "topological" examples and converse Hurewich theorems}\label{smash}

In this section we discuss the applications of our results to equivariant stable homotopy categories (as well as to general {\it purely compactly generated} weight structures). %So we 
 We significantly extend the main results of \cite[\S4]{bwcp}.

In \S\ref{smashs} we recall some properties of (smashing) weight structures generated by sets of compact objects in their hearts %(i.e., one should start from a negative subcategory of compact objects in $\cu$) 
and apply the %theory of this paper
 main results of the previous sections to this setting.
%relate representable pure functors to  killing weights.  %we also study the relation of this category $\aucp$ to detecting and killing weights.
We also study a certain cohomological dimension $1$ case separately. %later??!

In \S\ref{shur} we apply our results to prove some new properties of the spherical weight structure $\wg$ on the equivariant stable homotopy category $\shg$ of $G$-spectra ($\wg$ was introduced in \cite{bwcp}; actually, we  work in  a somewhat more general context).  We obtain a certain   improvement of the natural Hurewicz-type theorem in this setting. 
In particular, in the case $G=\{e\}$ (and so, $\shg=\shtop$) we give an if and only if condition for the vanishing of the negative singular homology groups of a spectrum. Moreover, our theory of objects without weights in a range gives certain canonical "decompositions" of spectra whose singular homology vanishes in two subsequent degrees (see Theorem \ref{tsh}(\ref{itopww}) and Remark \ref{rshg}(\ref{irghurww}) below).

\subsection{On purely compactly generated weight structures}\label{smashs}

In this section we will always assume that $\cu$ is closed with respect to (small) coproducts. %; a weight structure $w$ on it will always be smashing.
Let us recall a class of smashing weight structures that was studied in (\S3.2 of) \cite{bwcp} and (\S2.3 of) \cite{bsnew}.

\begin{defi}\label{dcomp}
\begin{enumerate}
Let $\hu$ be a full subcategory of a %(triangulated) category 
 $\cu$.

\item\label{idneg}
 We will say that the subcategory %$\hu\subset \cu$
 $\hu$  is {\it connective} (in $\cu$) if $\obj \hu\perp (\cup_{i>0}\obj (\hu[i]))$.\footnote{ In earlier text of the author connective subcategories were called {\it negative} ones. Moreover, in several papers (mostly, on representation theory and related matters) a connective subcategory satisfying certain additional assumptions was said to be {\it silting}; this notion generalizes the one of {\it tilting}.}

\item\label{idhull}
Let  $\hu'$ % of the %(obviously, additive) 
be the category %$\coprod \cp$  $\hu$
 of "formal coproducts" of objects of $\hu$,
%elements of $\cp$, 
i.e.,  the objects of %the latter category 
 $\hu'$ are of the form $\coprod_i P_i$ for (families of) $P_i\in \obj \hu$, and $\hu'(\coprod M_i,\coprod N_j)=\prod_i (\bigoplus_j \hu(M_i,N_j))$.
Then we will call $\kar(\hu')$ the {\it formal coproductive hull} of $\hu$.

\item\label{idlocal}
We will say  that a full triangulated subcategory $\du\subset \cu$ is  {\it localizing}  whenever $\du$ closed with respect to $\cu$-coproducts. 
%the class $\obj \du$ is $\al$-smashing (resp., smashing) in $\cu$.
Respectively, we will call the smallest localizing  subcategory of $\cu$ that contains a given class $\cp\subset \obj \cu$  the {\it  localizing subcategory of $\cu$ generated by $\cp$}. 

\item\label{idcomp}
An object $M$ of $\cu$ is said to be {\it compact} if %and only if 
 the functor $H^M=\cu(M,-):\cu\to \ab$ respects coproducts. 

\item\label{idcompg}
 We will say that $\cu$ is {\it compactly generated} by $\cp\subset \obj \cu$ if $\cp$ is a {\bf set} of compact objects 
that generates it as its own localizing subcategory.
%\item\label{idcgw}
\end{enumerate}
\end{defi}

Now let us recall the definition and the main properties of {\it purely compactly generated} weight structures; these are the ones that we will now describe.

\begin{theo}\label{thegcomp}
Let $\hu$ be a connective %(see Definition \ref{dwso}(\ref{idneg})) {dcomp}!!
 subcategory of $\cu$ such that $\cp=\obj \hu$ compactly generates $\cu$ (so, $\hu$ is small and its objects are compact in $\cu$). %; $\cp=\obj \hu$. 
 Then the following statements are valid.
\begin{enumerate}

\item\label{itnpb} $\cu$ is Karoubian and satisfies the  Brown representability property (see Definition \ref{dbrown}). %Proposition \ref{ppcoprws}(\ref{icopr5b})).

\item\label{itnp1}
There exists a unique smashing weight structure $w$ on $\cu$ such that $\cp\subset \cu_{w=0}$,  %this weight structure 
  and $w$ is left non-degenerate.

\item\label{itnp1d} For this weight structure $\cu_{w\le 0}$ (resp. $\cu_{w\ge 0}$) is the smallest subclass of $\obj \cu$ that is closed with respect to coproducts, extensions, and contains $\cp[i]$ for $i\le 0$ (resp. for $i\ge 0$), and $\hw$ %is equivalent to the big hull of $\cp$ (in the obvious way).
 %consists of all retracts of coproducts of (copies of) elements of $\cp$.
is %the coproductive 
 equivalent to the formal coproductive  hull of $\hu$  (see Definition \ref{dcomp}(1)) in the obvious way.

Moreover, $\cu_{w\ge 0}=(\cup_{i<0}\cp[i])^{\perp}$. % and for any $M\in \obj \cu$ any choice of its weight complex $t(M)$ is an object of the localizing subcategory of  $K(\hw)$ (cf. Proposition \ref{ppcoprws}(\ref{icoprhw})) generated by $\obj \hw$.

\item\label{itnp2} Let $H$ be a cohomological  functor from $\cu$ into an abelian category $\au$ that converts all small coproducts into products. Then it is pure if and only if it kills $\cup_{i\neq 0}\cp[i]$.

%+respects coproducts??!! 
\item\label{itnwe} Let $F:\cu\to \du$ be an exact functor that respects coproducts, where $\du$ is a triangulated category %that is closed with respect to small coproducts and 
 endowed with a smashing weight structure $v$. Then $F$ is weight-exact if and only if it sends $\cp$ into $\du_{v=0}$. %lwe,rwe??

\item\label{itnp3} The category $\hrt\subset \cu$ of $w$-pure representable functors from $\cu$ (so, we identify an object of $\hrt$ with the functor from $\cu$ that it represents)  is equivalent to the category $\au_{\cp}$ of %preadditive?!
 additive contravariant functors from $\hu$ into $\ab$ (i.e., we take those functors that respect the addition of morphisms).\footnote{%As explained 
According to Proposition 4.3.3 of \cite{bwcp}, the category $\hrt$ is actually the {\it heart of a $t$-structure} on $\cu$;  whence the notation.} 

%Thus this category
Moreover, $\aucp$ (and so also $\hrt$) is Grothendieck abelian, has enough projectives, and  an injective cogenerator; we will choose one and write $I$ for it. %and 

Furthermore, restricting functors representable by elements of $\cp$ to $\hw$ one obtains a %full embedding
  fully faithful functor $\cacp:\hw\to \aucp$ whose essential image is the subcategory of projective objects of $\aucp$.

\item\label{itnp4} %Let us write %$H_I$ for the representable functor $\cu(-,I)$ and $\cacp$ 
  %$\cacp$ for the Yoneda-type %embedding 
  % functor $Q\mapsto (P\mapsto \cu(P,Q)):\ \hw\to \aucp$ (where $P$ runs through $\cp$). Then the 
   The following assumptions on an object $M$ of $\cu$  are equivalent.

(i). $t(M)\in K(\hw)_{\wstu\ge 0}$. %??? (cf. Remark \ref{rwc}(\ref{irwco})).

(ii). $H_j^{\cacp}(M)=0$ for $j<0$ (here we use the notation from Theorem \ref{tpure}(3)).

(iii). $M\perp (\cup_{j<0}\{I[j]\})$.

\item\label{itnpnd} Assume that there exists an integer $j>0$ such that $\cp\perp \cup_{i\ge j}\cp[-i]$. Then $w$ is  non-degenerate. %, i.e., its weight-degenerate objects are zero. 
\end{enumerate}
\end{theo}
\begin{proof}
%Most of these statements 
Assertions \ref{itnpb}--\ref{itnp4} were mostly established in  \cite{neebook} and \cite{bsnew}; see \S3.2 of \cite{bwcp} for the detail.

Next, if $\cp\perp (\cup_{i\ge j}\cp[-i])$ then the description of $\cu_{w\ge 0}$ in assertion \ref{itnp1d} (along with the compactness of elements of $\cp$)  implies that $\cup_{i\ge j}\cp[i]\perp \cu_{w\le 0}$. Thus any right weight-degenerate element of $\cu$ belongs to $(\cup_{i\in \z}\cp[i])\perpp$.  Now, the latter class is zero since $\cu$ %equals its own localizing subcategory
 is compactly generated by $\cp$ (see the well-known Proposition 8.4.1 of \cite{neebook}); hence $w$ is right non-degenerate. %Since 
 
Lastly, $w$ is also left non-degenerate according to assertion \ref{itnp1}. %, $w$ is actually non-degenerate according to %Theorem \ref{tdegen}(II.1) (cf. also Theorem \ref{twkar}(I) along with  Proposition \ref{prwkarloc}(1).
\end{proof}

\begin{rema}\label{rnd}
1. There exist nice families of examples for part \ref{itnpnd} of our theorem. Firstly one can consider the {\it derived category} of a small {\it  differential graded category} $\bu$ (note that this construction gives all {\it algebraic triangulated categories}; see Theorem 3.8 of \cite{dgk}) such that for certain $j< 0$ we have $\bu^i(M,N)=\ns$ whenever $M,N\in \obj \bu$ and $i>0$ or $i<j$ (see loc. cit.). One can also consider the  derived category of $E$-modules where $E$ is a {\it commutative $S$-algebra} in the sense mentioned in Example 1.2.3(f) of \cite{axstab}; then one has to assume that $\pi_i(E)=\ns$ if $i<0$ (see \S7 of ibid. and Remark 4.3.4(2) of \cite{bwcp}) and if $i\gg 0$. %spectral?! infinity-examples?!

Thus one may say that the non-degeneracy of a purely compactly generated weight structure is a certain finiteness of the cohomological dimension condition. It may can be quite actual for representation theory (since it implies the conservativity of the weight complex functor and of certain weight-exact functors; see Theorem \ref{tdegen}(I.1)  and Proposition \ref{pwcons}). 

2. On the other hand, the spherical weight structure on $\shtop$ is right weight-degenerate (see Remark \ref{rshg}(\ref{irgsing}) below). Moreover,   the Chow weight structure for $\dm$ (as mentioned in Remark \ref{ray}(\ref{iray2}))  is right weight-degenerate as well (at least) if the  the base field  is "large enough"; %and the coefficient ring
  see Proposition 3.2.6 and Remark 3.2.7(2) of \cite{bwcp}. 

3. Proposition \ref{prwkarloc} easily implies that if the weight structure $w'$ as in part 3 or 4 of the proposition is non-degenerate then $F(M)=0$ under the corresponding assumptions on $M$; in the setting of %part 4 
 Remark \ref{rbinters} it follows that $M$ belongs to  $ \obj \cu_3$.

However, the author doubts that the condition $F(\cp)\perp \cup_{i\ge j}F(\cp)[-i]$ follows from $\cp\perp \cup_{i\ge j'}\cp[-i]$ (for any integer $j'>0$).
\end{rema}

Now let us relate purely compactly generated weight structures to the main definitions of the current paper.

\begin{coro}\label{cgcomp}
Adopt the notation and the assumptions of Theorem \ref{thegcomp}.

\begin{enumerate}
\item\label{ipc1}
%Then the 
The class of essentially $w$-positive objects coincides with $\perpp \cup_{j<0}\{I[j]\}$. It is also characterized by the vanishing of $H_j^{\cacp}(-)$ for $j<0$.

\item\label{ipc2}
%Then the 
The class of $w$-degenerate objects coincides with $\perpp \cup_{j\in \z}\{I[j]\}$ and also with $\perpp \cup_{j\in \z}\{\obj \hrt[j]\}$. %It is also 
 Moreover, this class is characterized by the vanishing of $H_j^{\cacp}(-)$ for all $j\in \z$.

\item\label{ipckw} A $\cu$-morphism $g$ kills weight $m$ (for some $m\in \z$) if and only if $H^m(g)=0$ for any pure representable (cohomological) functor $H$. 

\item\label{ipcww} An object $M$ of $\cu$ is without weights $m,\dots, n$ for some $m\le n\in \z$ if and only if 
$H^j(g)=0$ whenever $m\le j \le n$.

\item\label{ipc3}
%Then the class of $w$-negative objects coincides with 
$\cu_{w\le 0}=\perpp (\cup_{j>0}\obj \hrt[j])$. Moreover, this class is also annihilated by $H_i$ for all $i>0$ and for any $w$-pure homological functor $H$ from $\cu$.
%It is also characterized by the vanishing %\item\label{ipc4}
\end{enumerate}
\end{coro}
\begin{proof}
\ref{ipc1}.  According to Theorem \ref{twkar}(II), an object $M$ of $\cu$ is essentially $w$-positive if and only if $t(M)\in K(\hw)_{\wstu\ge 0}$. Combining this fact with   Theorem \ref{thegcomp}(\ref{itnp4}) we obtain our assertion.

\ref{ipc2}. Since $w$ is left non-degenerate (see Theorem \ref{thegcomp}(\ref{itnp1})), an object $M$ of $\cu$ is 
essentially $w$-positive if and only if it belongs to $\cu_{w\ge 0}$ (see Proposition \ref{prwkarloc}(1)). On the other hand, 
$M$ is weight-degenerate if and only if it is right weight-degenerate; % (see part 1 of that remark); 
 hence $M$ is weight-degenerate if and only if $M[j]$ is essentially $w$-positive  for all $j\in \z$. Hence our assertion follows from the previous one. 

\ref{ipckw}. Since $w$ is smashing and $\cu$ satisfies the  Brown representability property (see Theorem \ref{thegcomp}(\ref{itnpb},\ref{itnp1})), the assertion follows from Proposition \ref{puredetkw} immediately.

  \ref{ipcww}. This is a straightforward consequence of the previous assertion combined with %Proposition 
	 Theorem \ref{tpwckill}(\ref{iwckill1},\ref{iwckill3}) along with  Lemma \ref{lwwh}(\ref{irwc3}).

\ref{ipc3}. Since $w$ is left non-degenerate, $\cu_{w\le 0}$ coincides with the class of all essentially $w$-negative objects  according to Proposition \ref{prwkarloc}(1). Thus Theorem \ref{twkar}(I) (along with Remark \ref{rwkarloc})  gives the result in question. 
\end{proof}

Now let us demonstrate that one can say more on this setting if an additional assumption is imposed.

\begin{theo}\label{tshgp1}
Adopt the notation and assumptions of Theorem \ref{thegcomp}, and suppose in addition that the category $\aucp$ is of projective dimension at most $1$  (i.e., any its object has a projective resolution of length $1$). Let $m\le n\in \z$ and $g\in \cu(E,E')$.

Then the following statements are valid.

\begin{enumerate}
\item\label{itpwc1}
The category $\kw(\hw)$ equals  $K(\hw)$, and %the weight complex functor $t: \cu\to K(\hw)$ is exact. Moreover, 
 the natural functor $K(\hw)\to D(\aucp)$ is an equivalence.

%split the category?!
\item\label{itpwc2} For any objects $C$ and $C'$ in $D(\aucp)$ we have natural isomorphisms $C\cong \coprod H_j(C)[j]$ %define???!!
and $$D(\aucp)(C,C')\cong \prod_{j\in \z}\aucp(H_j(C),H_j(C'))\bigoplus %\coprod
(\prod_{j\in \z}\extaucp(H_j(C),H_{j+1}(C'))).$$
%All objects of $D(\aucp)$ split is direct sums of their homology, and  %$t(g)$ gives a well-defined class in $\aucp()$. %$\exto$?? whe to 0 when?!

\item\label{itpkw} $g$ kills weight $m$ if and only if $H_m^{\cacp}(g)=0$, the class of $t(g)$ in the group $\extaucp(H_{m-1}(t(E)),H_{m}(t(E')))$ (%natural cycle class for any abelian category???! 
here we use the identification provided by the previous two assertions) vanishes, and the morphism $H_{m-1}^{\cacp}(g)$ factors through a projective object of $\aucp$. % $H_n^{\cacp}(g)$ "lifts"?!
%(see the previous assertion) vanishes.

\item\label{itpww} $E$ is without weights $m,\dots,n$ (resp. $E\in \cu_{w\le m-1}$) if and only if %the $\aucp$-homology  (see assertion \ref{itpwc1})  $H_i(t(E))=0$ for $m\le i \le n$  and $H_{m-1}(t(E))$ is a projective object of $\aucp$.
$H_j^{\cacp}(E)=0$ for $m\le j \le n$ (resp. for $j\ge m$) and $H_{m-1}^{\cacp}(E)$ is a projective object of $\aucp$.
%\item\label{itpbab} and is a projective object of $\aucp$.
\end{enumerate}
\end{theo}
\begin{proof}
 \ref{itpwc1}. %All these 
These statements easily follow from Theorem \ref{thegcomp}(\ref{itnp3}) according to Remark 3.3.4 of \cite{bws} (that mostly relies on basic properties of derived categories; cf. (\ref{epres}) below). %; the corresponding argument %combined with some  is mostly basics on derived categories. % ; see  . %{vbook} %functorially splits as the coproduct of?! Cf. the proof of assertion {itpkw}??

\ref{itpwc2}. %Similarly to the argument needed to prove the previous assertion, one
The first splitting statement is well-known, %apply triangulated cohomological dimension 1 orthogonality?! use the projective description?!
 and one easily obtains the second one after noting that higher extension groups in the category  $\aucp$ vanish.

\ref{itpkw}. Once again, we should check whether $t(g)\backsim_{[-m,-m]} 0$ (see %Proposition 
 Theorem \ref{tpwckill}(\ref{iwckill1}))

Now we translate the splitting of $t(E)$ provided by previous assertions into  %the language?? 
\begin{equation}\label{epres}
t(E)\cong \coprod_{i\in \z}\co(A_{i}\stackrel {f_i}\to B_i)[i],\end{equation}
 where $f_i$ are certain $\hw$-morphisms %nado??
 that %become 
  are monomorphisms (between  projective objects) in $\aucp$; similarly, we consider
$t(E')\cong \coprod_{j\in \z}\co(A'_{j}\stackrel {f'_j}\to B'_j)[j]$. Obviously, all of the assumptions on $g$ considered in this assertion can be checked "summandwisely", i.e., we can and will assume that both $E$ and $E'$ are equal to single summands of this sort. % single length one complexes of this sort (i.e., that these summands are zero for all $i$ and $j$ besides a single pair of indices).

Next, all of our conditions are obviously fulfilled for these "morphisms of summands" if $\{i,j\}\not\subset \{m,m-1\}$. 
Moreover, we have $\co(A_{m}\stackrel {f_{m}}\to B_{m})[m]\perp \co(A'_{m-1}\stackrel {f'_{m-1}}\to B'_{m-1})[m-1]$ (since $f'_{m-1}$ becomes monomorphic in $\aucp$; one can make this calculation in $D(\aucp)$); thus any morphism between these two objects kills weight $m$ automatically. 

It remains to verify that the three remaining cases of $(i,j)$ give our three conditions on $g$. 

If $(i,j)=(m,m-1)$ then the previous assertion implies that  $$K(\hw)(t(E),t(E'))\cong  \extaucp (H_{m-1}^{\cacp}(E), H_{m}^{\cacp}(E')).$$ Moreover, since $t(g)\backsim_{[l,l]} 0$ for  all $l\neq -m$, $t(g)$ is zero (in $\kw(\hw)=K(\hw)$) if and only if $g$ kills weight $m$. %Since the group %$K(\hw)(\co(A_{m-1}\stackrel {f_{m-1}}\to B_{m-1})[m-1], \co(A'_{m}\stackrel {f'_{m}}\to B'_{m})[m])\cong \ext_{\aucp} (H_{m-1}^{\cacp}(E), H_{m}^{\cacp}(E'))$ (see assertion \ref{itpwc2}), we obtain that (
 Hence in this case $t(g)$ kills weight $m$ if and only if the class of $t(g)$ in %this extension group 
$\extaucp(H_{m-1}^{\cacp}(E), H_{m}^{\cacp}(E'))$ vanishes. 

Next, if $g$ kills weight $m$ then we have $H_{m}^{\cacp}(g)=0$ according to  %Corollary \ref{cgcomp}().  
 Proposition \ref {puredetkw}(I).
  Conversely, if $(i,j)=(m,m)$ then  $K(\hw)(t(E),t(E'))\cong  \aucp (H_{m}^{\cacp}(E), H_{m}^{\cacp}(E'))$. Thus if $H_{m}^{\cacp}(g)=0$ then $t(g)=0$ in this case; hence $t(g)\backsim_{[-m,-m]} 0$. %$g$ kills weight $m$.

It remains to consider the case $(i,j)=(m-1,m-1)$. By definition, $t(g)\backsim_{[-m,-m]} 0$ if and only if % it %equals (in $K(\hw)$) a morphism $(h^i):t(E)\to t(E')$ with $h^{-m}=0$
 $t(g)$ factors through the obvious morphism $B'_{m-1}[m-1]\to t(E')$. The latter condition clearly implies that $H_{m-1}^{\cacp}(g)$ factors through $B'_{m-1}$. Conversely, if $H_{m-1}^{\cacp}(g)$ factors through a projective object of $\aucp$ then $t(g)$ factors through an object of $\hw[m-1]$ (note that $\aucp$ embeds into $K(\hw)_{\wstu\ge m-1}$ via projective resolutions). Thus $t(g)$ factors through $B'_{m-1}[m-1]$ according to Proposition \ref{pbw}(\ref{ifact}) (applied to $K(\hw)$ endowed with the stupid weight structure).

Assertion \ref{itpww} follows from the previous one more or less easily. %Arguing as in the proof of Corollary \ref{cgcomp}(\ref{ipcww}) we obtain that 
Combining Theorem \ref{twkar}(II) (see condition \ref{icwkwp} in it) with %Corollary \ref{cwkar} 
 Proposition \ref{prwkarloc}(1) we obtain that  $E$ belongs to $ \cu_{w\le m-1}$ if and only if $\id_E$ kills weight $i$ for all $i\ge m$.  Recall also that $E$ is without weights $m,\dots,n$ if and only if $\id_E$ kills weight $i$ for $m\le i \le n$ (see Theorem \ref{tprkw}(\ref{iprkws})). Hence it remains to note that for any $i\in \z$ the morphism $H_i^{\cacp}(\id_E)$  clearly equals the corresponding identity, whereas the classes of $t(\id_E)$ in all the groups $\extaucp(H_i(E),H_{i+1}(E))$ (see assertion \ref{itpwc2}) are zero.
%We note that Immediate from the previous assertion. {tprkw} {iprkw1}
\end{proof}

\begin{rema}\label{rtate}
1. Let us describe a motivic example to our theorem.

Inside the category $\dm$ of  motives (with integral coefficients) over any perfect field % (either the effective or stable one; 
(see \S4.2 of \cite{degmod} or \cite[\S1.3,\S3.1]{bokum})  one can take  $\cu$ to be its  localizing subcategory $DTM$  generated by the Tate motives $\z(i)$ for $i\in\z$. % (i.e., by its smallest triangulated subcategory containing all $\z(i)$ and closed with respect to small coproducts).  %Since motives of the form $\z(i)[2i]$ (for $i\in\z$) are Chow ones, the Chow weight structure on 

%This category possesses two important weight structures: the heart of the first (''Chow") one is %generated by 
 %the big hull of $\z(i)[2i]$, %(i.e. consists of retracts of coproducts of families of $\z(i)[2i]$; actually, retracts are not necessary here)
%and this weight structure is "compatible" with Chow weight structure for the whole $\dm$ (see Remark 3.1.2 %of \cite{bws} or of \cite{bokum});   the second ("Gersten"; cf. Proposition 5.2.6 of \cite{bgn}) heart is %weakly???? generated by $\z(i)[i]$. 
	
	Now,  the category $\cp=\{\z(i)[2i],\ i\in \z\}$ is connective in $\dm$, %we have 
	$\z(i)[2i]\perp \z(j)[2j]$ for any $i\neq j\in \z$, whereas  $\dm(\z(i)[2i],\z(i)[2i])\cong \z$ for all $i$;  %immediately from 
	 see Corollary 6.7.3 of \cite{bev}. % we obtain that
	 Hence %for $\cp=\{\z(i)[2i],\ i\in \z\}$
	 we have $\aucp\cong \ab^{\z}$ in this case; thus $\aucp$ is of projective dimension $1$.
	
	Moreover, one can take any Dedekind domain or a field for the  coefficient ring in this example.

2. The case $E\in \cu_{w\le m-1}$ in part \ref{itpww} of our theorem is an abstract generalization of Proposition 6.16 %6.3.16 %
of \cite{marg}  %Proposition 6.8 of \cite{christ} 
(where $\cu=\shtop$ was considered; see Theorem  \ref{tsh} below). Respectively, Proposition 6.17 of loc. cit. is the corresponding case of the orthogonality axiom for the weight structure $\wsp$. Note however that the methods of loc. cit. do not seem to extend to the setting of Theorem \ref{tshg} (for a general $G$) and Theorem \ref{thegcomp}. Moreover, Proposition 7.1.2(a) of \cite{axstab} essentially gives the existence of weight Postnikov towers corresponding to a certain weight structure $w$ only for $w$-bounded below objects (cf. Remark 4.3.4(2) of \cite{bwcp}).

3. The author suspects that the functor $t: \cu\to K(\hw)$ is actually exact in this case. Most probably this conjecture can be proven similarly to Corollary 3.5 of \cite{sosnwc}.
\end{rema}

\subsection{On equivariant spherical weight structures and stable Hurewicz theorems}\label{shur}

Now let us generalize Theorem  4.1.1 of \cite{bwcp} %advise to consult??
 to the setting of $\lam$-linear equivariant stable homotopy categories and add several new statements. We will need some notation.

\begin{itemize}
\item  We choose  a set of prime numbers $S\subset \z$; denote $\z[S\ob]$ by $\lam$. $S$ may be empty; then $\lam=\z$ (and the reader may confine herself to this important case). %We will usually omit $S$ (and $\lam$) in the notation.

\item $G$ will be a (fixed) compact Lie group; we will write $\shg$ for the stable homotopy category of $G$-spectra indexed by a complete $G$-universe.

\item Denote the subcategory of  $\lam$-linear objects of $\shg$ %will be denoted 
 by $\shgl$. According to Proposition A.2.8 of \cite{kellyth}, $\shgl$ is a triangulated subcategory of $\shg$ and there exists an exact left adjoint $\ls=(-)[S\ob]$ to the embedding $\shgl\to \shg$. 

\item We take $\cp$ to be the set of spectra of the form %$G/H^+$
$\ls(S_H^0)$, where $H$ is a closed subgroup of $G$ (cf. Definition I.4.3 of \cite{lms}; recall that $S_H^0$ is constructed starting from the $G$-space $G/H$). We will %use the notation
 write  $\hu$ for the corresponding (preadditive) subcategory of $\shgl$. Respectively, if $\lam=\z$ then $\hu$  is the (stable) {\it  orbit category} of ibid. %p. 238??

Recall also that $S_H^0[n]\in \obj \shg$ is the corresponding sphere spectrum $S_H^n$ essentially by definition (see loc. cit.). 

\item The equivariant homotopy groups of an object $E$ of $\shg\supset \shgl$ are defined as $\pi_n^H(E)=\shg(S_H^n,E)$ (for all $n\in \z$; see %Definition I.4.4(i) and \S I.5
\S I.6  and Definition I.4.4(i) of ibid.). 

\item We will write $\emg$ for the full subcategory of $\shgl$ whose object class is $(\cup_{i\in \z\setminus \ns} \cp[i])^{\perp}$ (in the case $\lam=\z$ the  objects of $\emg$ are {\it Eilenberg-MacLane} $G$-spectra; see \S XIII.4 of \cite{mayeq}).

\item We will write $\macg$ for the category of %preadditive
 additive contravariant functors from $\hu$ into $\ab$ (cf. Theorem \ref{thegcomp}(\ref{itnp3})); we will call its objects %are the %??? 
{\it Mackey functors} (cf. loc. cit.). Respectively, $\cacp$ will denote the Yoneda embedding%????
  $\hu\to \macg$. 

\item We will call $\perpp (\cup_{i\in \z}\obj \emg[i])\subset \obj \shgl $ the class of {\it acyclic} spectra (i.e., a spectrum is acyclic if it is annihilated by $H^i$ for all $H$  represented by  $\lam$-linear Eilenberg-MacLane spectra and $i\in \z$).
%say that a a $G$-spectrum $E$ is {\it acyclic} if 

\item Let us now give references to those weight structure definitions that are necessary to understand the theorem below. %Weight structures are defined in 
 The most basic of those are Definitions \ref{dwstr} and \ref{dwso}(\ref{idh},\ref{id=i},\ref{ilrd},\ref{iwnlrd}); pure functors are defined in Definition \ref{dpure}; weight-degenerate objects are defined  in Definition \ref{dwdegen},    essentially weight-positive objects are defined  in %Definition \ref{dewpn}
 Definition \ref{dpkar}(4); objects without weights $m,\dots,n$ (along with morphisms killing weights) are defined in Definition \ref{dkw}. 
\end{itemize}

Now we demonstrate that our theory gives several interesting properties of $\shg\supset \shgl$ (and one can assume that $S$ is empty, $\lam=\z$, and $\shgl=\shg$). The proofs of our statements are rather easy consequences of the general theory. %????; they also have much in common with each other.

\begin{theo}\label{tshg}
Let $n\in \z$, $h\in \cu(E,E')$ (for $E$ and $E'$ being  objects of $\shg$). %closed subgroup $H$ of $G$??? 
Then the following statements are valid.

\begin{enumerate}
\item\label{itgl} One can assume that the functor $\ls$ is identical on the subcategory $\shgl$ of $ \shg$. Moreover,  for any closed subgroup $H$ of $G$ and an object $C$ of $\shg$ %, and $n\in \z$ 
 we have  $\shg(\ls(S_H^n),\ls(C)) =\shgl(\ls(S_H^n),\ls(C))\cong \shg(S_H^n,\ls(C))\cong \pi_n^H(C)\otimes_{\z}\lam$.

\item\label{itgfine}
The category $\cu=\shgl$ and the class  $\cp$ specified above %being the set of spectra of the form 
 satisfy the assumptions of Theorem \ref{thegcomp}. Thus $\cp$ gives a weight structure $\wgl$ on $\shgl$ whose heart consists of retracts of coproducts of elements of $\cp$ and equivalent to the  formal coproductive hull of $\hu$. %\cpg, \hug?!!

\item\label{itgconn}
The class of $n-1$-connected $\lam$-linear spectra  (see Definition I.4.4(iii) of \cite{lms}; i.e., this is the class $(\cup_{i<n}\cp[i])\perpp$) coincides with $\shgl_{\wgl\ge n}$. In particular, $\shgl_{\wgl\ge 0}$ is the class of $\lam$-linear {\it connective} spectra. % and $\wg$-bounded below objects are the %connected  bounded below spectra of loc. cit.

Moreover, $\shgl_{\wgl\ge 0}$ is the smallest class of objects of $\shgl$ that contains $\cup_{i\ge 0}\cp[i]$ %for $i\ge 0$ 
 and is closed with respect to coproducts and extensions.
 
\item\label{itgpure} A (co)homological functor from $\shgl$ into an abelian category $\au$ that respects coproducts (resp. converts them into products) is $\wgl$-pure if and only if it kills $\cup_{i\neq 0}\cp[i]$. %?? (cf. the dimension axiom (XIII.4.4) of \cite{mayeq}). 

In particular, all %Eilenberg-MacLane $G$-spectra 
 objects of $\emg$ represent $\wgl$-pure functors. %Thus the cohomology functors $H_G^0(-,M)$ (for a Mackey functor $M$) defined in \S XIII.4 of ibid. also are. The same thing: R?!
%The pure functor %standard homology?! "172"

\item\label{itgemg} The category $\emg$ is naturally %isomorphic 
  equivalent to $\macg$ in the obvious way; thus $\emg$ is Grothendieck abelian and has an injective cogenerator $I$.

\item\label{itgdeg} 
The following  conditions are  equivalent.

(i). $E$ is acyclic.

(ii). $E$ is %weight
$\wgl$-degenerate.

(iii). $E$ is right %weight
 $\wgl$-degenerate.

(iv). $E\perp \shgl_{\wgl\ge i}$ for any $i\in \z$.

(v).  $H^{\cacp}_i(E)=0$ for any $i\in \z$.

(vi). $E\perp (\cup_{i\in \z}I[i])$.

(vii). $E$ is annihilated by any $\wgl$-pure homological functor from $\shgl$.

\item\label{itgkw} The following statements are  equivalent as well.

(i). $h$ kills weight $0$. 

(ii). $H(h)=0$ for any  $\wgl$-pure homological functor $H$ from $\shgl$.

(iii). $P(h)=0$ for any  $\wgl$-pure cohomological functor $P$ from $\shgl$.

(iv). $P(h)=0$ for any  %pure cohomological 
 functor %of the form
 $P=\shgl(-,J)$, where $J\in \obj \emg$.

\item\label{itghur} The following  conditions are  equivalent.

(i). $E$ is essentially $\wgl$-positive.

(ii). $E$ is an extension of a connective spectrum by an acyclic one.

(iii). $H_j(E)=0$ for any  $\wgl$-pure homological functor $H$ from $\shgl$ and $j<0$.

(iv). $P^j(E)=0$ for any  $\wgl$-pure cohomological functor $P$ from $\shgl$ and $j<0$.

(v). $E\perp I[j]$ for all $j<0$.

(vi). $H_j^{\cacp}(E)=0$ for all $j<0$.

\item\label{itgneg} $ \shgl_{\wgl\le 0}$  is the smallest subclass of $\obj \shgl$ that is closed with respect to coproducts, extensions, and contains $\cp[i]$ for $i\le 0$. This class also equals  $\perpp (\cup_{i\ge 0}\obj \emg[i])=   \perpp \shgl_{\wgl\ge 1}$;\footnote{Recall here $\shg_{\wgl\ge 1}$ is the class of $0$-connected $\lam$-linear $G$-spectra.} %this class is also 
 moreover, it is annihilated by $H_i$ for all $i>0$ and for any pure homological functor $H$ from $\shgl$.

% (cf. assertion \ref{itgconn} for the discussion of the class.

%\item\label{itgbo} %which bounds?! $I$-ones separately; Proposition \ref{pgcomp}(\ref{ipc1}). Put in R?!

%????\item\label{itghur}

\item\label{itgww} Let $m\le n\in \z$. Then the following  conditions are  equivalent.

(i). $E$ %is without weights 
 avoids weights $m\dots n$.

(ii). There exists a distinguished triangle $E_1\to E\to E_2\to E_1[1]$ such that $E_1\in \shgl_{\wgl\le m-1}$ and $E_2\in \shgl_{\wgl\ge n+1}$. Moreover, if this is the case then this triangle is canonically determined by $E$.

(iii).   $E\perp  (\cup_{m\le i \le n}\obj \emg[-i])$.

(iv).  $E$ is  annihilated by $H_{i}$ whenever $m\le i\le n$  and $H$ is a  pure homological functor from $\shgl$.
\end{enumerate}
\end{theo}
\begin{proof}
\ref{itgl}. $\ls$ can be assumed to be identical on $\shgl$ since it is left adjoint to the embedding of $\shgl$ %is a subcategory of 
 to $\shg$. Next, the only non-trivial isomorphism for morphism groups in the assertion is given by Corollary A.2.13  of \cite{kellyth}.

\ref{itgfine}. The fact that objects of the form $S_H^0$ form a connective subcategory of $\shg$ that compactly generates this category
%compactness of $S_H^0$ in $\shg$  for any closed subgroup $H$ of $G$ along with the negativity of the subcategory formed by objects of this type 
can be deduced from %Lemma I.5.3 of 
the results of \cite{lms}; see (the proof of) \cite[Theorem 4.1.1]{bwcp} for more detail. Thus applying the previous assertion we immediately obtain that $\hu$ is connective in $\shgl$. It remains to note  that the functor $\ls$ sends any  family of compact generators of $\shg$ into a one for $\shgl$ according to  Proposition A.2.8 of \cite{kellyth}.
Lastly, we apply Theorem \ref{thegcomp}(\ref{itnp1d}) to compute the heart of $\wgl$.

%\ref{itgrest}. %Applying 
 %Recall that $\shg(S_H^n,C) \cong \shgl(\ls(S_H^n,\ls(C)))$ (see assertion \ref{itgl}); thus it remains to recall the description of 
%apply  the corresponding case of 
 %$\shgl_{\wgl\ge n}$ provided by Theorem \ref{thegcomp}(\ref{itnpr}).

\ref{itgconn}. By definition, a $G$-spectrum $N$ is $n-1$-connected whenever $\pi_i^H(N)\cong \shg(S_H^i,N)= \ns$ for all $i<n$ and $H$ being any closed subgroup of $G$. Since $\shg(S_H^n,N) \cong \shgl(\ls(S_H^n),\ls(N))$ (see assertion \ref{itgl}), %Hence 
 it remains to apply Theorem \ref{thegcomp}(\ref{itnp1d}) to obtain all the statements in question.

Assertion \ref{itgpure} follows from Theorem \ref{thegcomp}(\ref{itnp2}) immediately. % (cf. also Remark \ref{rgensmash}(\ref{irspure})).
Moreover, assertion \ref{itgemg} is just the corresponding case of Theorem \ref{thegcomp}(\ref{itnp3}).

\ref{itgdeg}. According to assertion \ref{itgfine}, we can apply Corollary \ref{cgcomp}(\ref{ipc2}) to obtain that conditions 
(i), (ii), (v), and (vi)  are equivalent. Next, conditions (ii) and (iii) are equivalent since $\wgl$ is left non-degenerate (see Theorem \ref{thegcomp}(\ref{itnp1}) and Proposition \ref{prwkarloc}(1)), and applying Proposition \ref{pbw}(\ref{iort}) we obtain the equivalence of (iii) and (iv). Lastly, conditions (ii) and (vii) are equivalent by Theorem \ref{twkar}(I) (see condition I.\ref{icwpuh} of the theorem).

\ref{itgkw}. Conditions (i), (ii), and (iv) are equivalent according to Proposition \ref{puredetkw}. Next, conditions (ii) and (iii) are equivalent since for any  homological functor $H:\shgl\to \au$ the corresponding cohomological functor $P$ from $\shgl$ into $\au\opp$ is pure if and only if $H$ is.

\ref{itghur}. Since $\shgl$ is Karoubian,  %(see Remark \ref{rlcg}), 
 conditions (i) and (ii) are equivalent according to Theorem %\ref{tdegen}(I.3,II.1) 
 \ref{twkar}(II) (combined with assertions \ref{itgconn} and \ref{itgdeg}). Moreover, (i) is equivalent to (v) and (vi) by Theorem \ref{thegcomp}(\ref{ipc1}). 

Lastly, condition (iv) clearly implies condition (v), and the obvious argument that we have just used yields that conditions (iii) and (iv) are equivalent.
% It remains to recall that (i) is equivalent to (iv) according to Proposition \ref{puredetkw}(II). 

\ref{itgneg}. The first of these descriptions of $ \shgl_{\wgl\le 0}$ is given by Theorem \ref{thegcomp}(\ref{itnpb}). Next,
  $ \shgl_{\wgl\le 0}=   \perpp \shgl_{\wgl\ge 1}$ according to  Proposition \ref{pbw}(\ref{iort}). It remains to apply  Theorem \ref{thegcomp}(\ref{ipc3}) %ecall the definition of pure functors????
	to conclude the proof. 

\ref{itgww}. By definition, $E$ is without weights $m \dots n$ if and only if $\id_E$ kills these weights. Hence applying assertion \ref{itgkw} we obtain the equivalence of conditions (i), (iii), and  (iv). Lastly, conditions (i) and (ii) are equivalent according to Theorem \ref{tprkw}(\ref{iwildef1},\ref{iwildef2}).
\end{proof}

\begin{rema}\label{rshg}
\begin{enumerate}
\item\label{irghur} Let us explain that part \ref{itghur} of our theorem is a certain %converse for
 "unbounded improvement" of the natural Hurewicz-type theorem for this context.

We recall that the "usual" Stable  Hurewicz Theorem essentially says that in the case  $G=\{e\}$ (and $\lam=\z$; so, $\shg=\shtop$)  %essentially says that 
   a $\wgl=\wsp$-bounded below spectrum $E$ is connective if and only if its singular homology is concentrated in non-negative degrees. A certain equivariant version of this statement is given by Theorem 2.1(i) of \cite{lewishur} (cf. also Theorem 1.11 of ibid. and Proposition 7.1.2(f) of \cite{axstab}); note that one replaces singular homology by $H^{\cacp}$ in it (cf. part \ref{irgsing} of this remark). %Another version of 

%Thus 
%Recall now 
Now, it is easily seen that those essentially $\wgl$-positive objects that are $\wgl$-bounded below are connective. %(see  %Lemma \ref{ldeg}(3)).
  Hence part \ref{itghur} of our theorem naturally generalizes %statements of this 
 the aforementioned equivariant Hurewicz-type theorem to arbitrary objects of $\shgl$. Our generalization depends on the notion of acyclic spectra, and the corresponding part \ref{itgdeg} of our theorem also appears to be quite new (cf. also part \ref{irgsing} of this remark). 

\item\label{irghurww}
The notions of killing weights and %being without 
 avoiding weights $m,\dots,n$ (along with parts \ref{itgkw}, \ref{itgww}, and \ref{itgneg}  of our theorem) appear to be new in this context (even when restricted to the case $\cu=\shtop$) as well. In particular, we obtain %that one can obtain 
 certain canonical "decompositions" of spectra (see condition \ref{itgww}.(ii) in our theorem) %, recall that $\shgl$ is Karoubian, and apply ) 
by looking at their cohomology with coefficients in Eilenberg-Maclane spectra. This statement becomes especially nice when applied to $\shtop$; see Theorem \ref{tsh}(\ref{itopww}) below. 

\item\label{itgcacp} Recall that if $\lam=\z$ then the pure homological functor $H^{\cacp}$ is the %{\it ($RO(G)$-graded) 
 equivariant    ordinary homology with Burnside ring coefficients  functor $H^G_0$ considered in \cite{lewishur} (cf. also Definition X.4.1 of \cite{mayeq}), and for any Mackey functor $M$ the corresponding pure functor $H_M$ coincides with $H_G^0(-,M)$  in Definitions X.4.2 and \S XIII.4 of ibid.  Clearly, it follows that the functors  $H^{\cacp}$ and $H_M$ are closely related to these notions as well (for a general $\lam$ also). 

\item\label{irgsing} None of the descriptions of acyclic spectra in part \ref{itgdeg} of our theorem characterizes them "explicitly".
In the case $G=\{e\}$ our definition of this notion coincides with the one considered in \cite{marg}; see %Theorem \ 
Theorem  \ref{tsh}(\ref{itopsingh}) below. At least, this gives an explicit example of a non-zero acyclic spectrum (see Theorem 16.17 of \cite{marg}).

 One can possibly say more on %them 
acyclic spectra for $G\neq \{e\}$ via considering exact functors connecting equivariant stable homotopy categories corresponding to distinct groups. Note that some of these functors are weight-exact, and one can also apply Proposition \ref{puredetkw}(II) for treating "induced" (co)homology functors. 

The author conjectures that a spectrum $E$ is acyclic in $\shg$ if and only if $\Phi^{H}\circ i_H^*(E)\in \obj \shtop$ (where $i_H:H\to G$ runs through all inclusions of closed subgroups; see \S II.4 and  of Definition II.9.7 of \cite{lms}) is. % %varphi?! 
%?????!!!!!!!This would certainly be true if for any $H$ of this sort we have $H^{\cacp}_*(E)(S_H^0)\cong \hsing_*(\Phi^{H}\circ i_H^*(E))$. Note here that it suffices to verify this fact for $E\in \cp$. The author conjectures that this restricted case of the latter statement follows from Proposition V.9.6 of ibid. whenever $G$ is finite, and that the statement is wrong in general.\footnote{Recall also that loc. cit. gives a "classical" description of the orbit category $\hu$ (when $G$ is finite); this gives a  simple description of the category $\macg$.}

\item\label{irglam} Theorem \ref{tshgp1} certainly gives some more information on ("weights of")  objects of $\shgl$ whenever the corresponding category $\macg$ is of projective dimension at most $1$. So we note that this assumption is fulfilled whenever $G$ is a finite group of order $n$ and $1/n\in \lam$ (one should join  Theorem 2.1 of \cite{green} with the finite projective dimension statement established in \S6 of ibid. to obtain this fact). 

\item\label{irguniv}
It appears that the statements contained in  \cite{lms} are actually sufficient to generalize all the assumptions of our theorem %can be easily extended 
 to the case where $\cu$ is the stable homotopy category of $G$-spectra indexed on a not (necessarily) complete $G$-universe. In particular, %if the corresponding 
 this universe can be $G$-trivial (i.e., $G$ acts trivially on its spaces); this %corresponds to the 
 allows us to apply it to the corresponding representable functors as considered in \S IV.1 of \cite{bred}.

\item\label{irgother} Let us mention some other nice properties of $\wg$.

Firstly, it restricts to the subcategory of compact objects of $\shgl$; cf. Theorem 4.1.1(2) of \cite{bwcp}. Next, the class $\shgl_{\wg\ge 0}$ can also be described as $\shgl_{t\ge 0}$ for a certain {\it Postnikov $t$-structure} $t$ (see Definition 4.3.1(I) and Proposition 4.3.3 of ibid.); yet this $t$-structure does not restrict to compact objects of $\shgl$.

Note also that our theory provides a certain inverse Hurewicz-type theorem (and also other "decompositions" as well as several  new definitions; see part \ref{irghurww} of this remark) for the so-called  connective stable homotopy theory as discussed in \S7 of \cite{axstab}; see Remark 4.3.4(2) of \cite{bwcp} for more detail.
\end{enumerate} %cellular towers?!
\end{rema}

Now we apply our results to the %("topological") 
 stable homotopy category $\shtop$ (whose detailed description can be found in \cite{marg}); some of these statements were already stated in (Theorem 4.2.1 of) \cite{bwcp}. This corresponds to the case of a trivial $G$ and $S$ in Theorem \ref{tshg}; so we will write $\emo$ for $\emg$ and $w^{sph}$ for $\wgl$ in this case (and  use the remaining notation from this theorem). %; recall also there is a list of references to the main weight structure definitions of this paper preceding that theorem).

\begin{theo}\label{tsh}
Set $\cp=\{S^0\}$; assume $m\le n\in \z$ and $g:E\to E'$ is an $\shtop$-morphism. %let $m\le n\in \z$ and $E\in \obj \shtop$.

 Then the following statements are valid.

\begin{enumerate}
\item\label{itoptriv}
The functor $\shtop(S^0,-)$ gives equivalences $\hwsp\to \abfr$ (the category of free abelian groups) and $\emo\to \ab$; thus $\aucp$ is equivalent to $\ab$ as well.

%Moreover, $\wsp$ restricts to the category of finite spectra, and the heart of this restriction is equivalent to the category of free finitely generated abelian groups.

\item\label{itopsingh}
The functor $\cacp$ is essentially the singular homology functor $\hsing$; respectively, acyclic spectra in $\shtop$ are characterized by the vanishing of their singular homology (cf. \S6.2 of \cite{marg}).

\item\label{itopsingc}
For any abelian group $\gam$ and the corresponding spectrum $\egam\in \obj \emo$ the functor $\shtop(-,\egam)$ is isomorphic to  the singular cohomology with coefficients in $\gam$ one.

\item\label{itopkw} $g$ kills weight $m$ if and only if $\hsing_m(g)=0$, the class of $t(g)$ in the group $\ext_{{\ab}}(\hsing_{m-1}(E),\hsing_m(E'))$ (see Theorem \ref{tshgp1}(\ref{itpkw})) %as given by here we use the identification provided by the previous two assertions) 
vanishes, and the morphism $\hsing_{m-1}(g)$ factors through a  free abelian group. %projective object of $\aucp$. 

This condition is also equivalent to the vanishing of $\hsingc^m(-,\gam)(g)$ for any abelian group $\gam$.

\item\label{itopww} %The following conditions are equivalent as well.(i)
 $E$ is an extension of an $n$-connected spectrum by an element of $ \shtop_{\wsp\le m-1}$ if and only if 
$\hsing_j(E)=0$ for $m\le j \le n$  and $\hsing_{m-1}(E)$ is a free abelian group.
%without weights $m,\dots,n$
%here we use the identifications provided by the previous two 

\item\label{itopass} The following conditions are equivalent.

 (i) $E\in \shtop_{\wsp\le n}$;

(ii) $\hsing_i(E)=\ns$ for $i>n$ and $\hsing_n(E)$ is a free abelian group;

(iii)  $\hsingc^i(E,\gam)=\ns$ for any $i>n$  and any abelian group $\gam$;

(iv) $E$ is an {\it$n$-skeleton} (of %a certain 
 some spectrum) in the sense of %ibid. %
\cite[\S6.3]{marg} (cf. also Definition 6.7 of \cite{christ}).

\item\label{itopcell} %a wpf??????!!
A %weight 
{$\wsp$-Postnikov tower} (see Definition 1.3.3 of \cite{bwcp}) of  a spectrum $X\in \obj \shtop$ is the same thing as a cellular tower for $E$ in the sense of %the beginning of  \S6.3 of 
 \cite[\S6.3]{marg}.

\item\label{itophur} $E$ is an extension of a connective spectrum by an acyclic one if and only if $\hsing_j(E)=\ns$ for all $j<0$; this is also equivalent to the vanishing of $H_{sing}^j(E,\q/\z)$  for all $j<0$.
\end{enumerate}

\end{theo}
\begin{proof}
Assertions \ref{itoptriv}--\ref{itopsingc} %\ref{itopsingh} 
 easily follow from Theorem \ref{thegcomp}; they are also contained in Theorem 4.2.1 of \cite{bwcp}.
These facts also yield that Theorem \ref{thegcomp}(\ref{itghur})  implies our assertion \ref{itophur}.

Next, the category $\ab\cong \aucp$ is of  cohomological dimension $1$; hence we can combine Theorem \ref{tshgp1} with the preceding assertions along with %the corresponding parts of  
Theorem \ref{thegcomp}(\ref{itgkw},\ref{itgneg}) to obtain assertions \ref{itopkw},  \ref{itopww}, %(cf. Theorem \ref{tshg}(\ref)),
   and the equivalence of conditions (i)--(iii)  in assertion \ref{itopass}. Moreover, the latter equivalence statements implies that these conditions  are fulfilled if and only if $E$ is an $n$-skeleton, and also that assertion \ref{itopcell} is valid according to Theorem 4.2.1(4,5) of \cite{bwcp}. 
	%are equivalent to  
%Alternatively, one may apply the aforementioned Proposition 6.16  of \cite{marg} with Theorem \ref{tshgp1}(\ref{itpww})???!!
\end{proof}

%???\begin{rema}\label{rmarg}

%adjacent t: remark?!

\end{document}